\documentclass[final,envcountsect]{svjour3}

\usepackage{graphics,epsfig,amssymb}
\usepackage{amsmath}
\usepackage{amsfonts}
\usepackage{amssymb,latexsym}
\usepackage[all]{xy}
\usepackage{graphicx,color}
\usepackage[mathscr]{eucal}
\usepackage{verbatim}
\usepackage{hyperref}
\usepackage{enumitem} 
\usepackage[notref,notcite]{showkeys}
\usepackage{authblk}
\usepackage{xcolor}
\usepackage{todonotes}
\smartqed
\newtheorem{thm}{Theorem}[section]
\newtheorem{lem}{Lemma}[section] % following theorem in numbering

\newcommand{\R}{{\mathbb R}}
\newcommand{\N}{{\mathbb N}}

\newcommand{\Z}{\mathbb Z}
\newtheorem{thmx}{Theorem}

\begin{document}

%opening
\title{Global dynamics of a novel delayed logistic equation arising from cell biology}

\author{Ruth E. Baker        \and
	Gergely R\"ost %etc.
}


\institute{Ruth E. Baker \at
	Mathematical Institute, University of Oxford, UK            %  \\
	\and
	Gergely R\"ost \at
	Mathematical Institute, University of Oxford, UK\\
	\& Bolyai Institute, University of Szeged, Hungary\\
	\email{rost@math.u-szeged.hu}}
%\author{Ruth Baker \& Gergely R\"ost}
\date{\today}
\maketitle

\begin{abstract} 

 The delayed logistic equation (also known as Hutchinson's equation or Wright's equation) was originally introduced to explain oscillatory phenomena in ecological dynamics. While it motivated the development of a large number of mathematical tools in the study of nonlinear delay differential equations, it also received criticism from modellers because of the lack of a mechanistic biological derivation and interpretation. Here we propose a new delayed logistic equation, which has clear biological underpinning coming from cell population modelling. This nonlinear differential equation includes terms with discrete and distributed delays. The global dynamics is completely described, and it is proven that all feasible nontrivial solutions converge to the positive equilibrium. The main tools of 
the proof rely on persistence theory, comparison principles and an $L^2$-perturbation technique. Using local invariant manifolds,
a unique heteroclinic orbit is constructed that connects the unstable zero and the stable positive equilibrium, and we show that these three complete orbits constitute the global attractor of the system. Despite global attractivity, the dynamics is not trivial as we can observe long-lasting transient oscillatory patterns of various shapes. We also discuss the biological implications of these findings and their relations to other logistic type models of growth with delays.
\end{abstract}

global attractor, stability, long transient.

\medskip
{\bf Mathematical Subject Classfication:}  92D25, 34K20.

\baselineskip=18pt

%====================================================================

\section{Introduction}

The well known logistic differential equation $N'(t)=r N(t)\left(1-N(t)/K\right)$ was proposed by Verhulst in 1838 to resolve the Malthusian dilemma of unbounded growth \cite{bacaer}. Here $N(t)$ represents the population density at time $t$, $r$ is referred to as the intrinsic growth rate, and $K$ as the carrying capacity. Generally, such saturating population growth can be achieved by assuming that an increase in the population size leads to a decrease in fertility or an increase in mortality. This is reasonable when resources are limited: if the population size exceeds some critical level, the habitat cannot support further growth. The logistic equation has seen widespread use in modelling studies across biology with applications in single bacteria, cell and animal populations, and also interacting populations, cancer, and infectious diseases. 

In the framework of the logistic ordinary differential equation, the population always converges to a stable steady state. However, in the context of ecology, oscillatory behaviours have also been observed, even in the absence of external periodic forcing. For example, in Daphnia populations, oscillations are possible because the fertility
of females depends not merely on the actual population density, but also on the past densities to which it has been exposed. This motivated Hutchinson in 1948  \cite{hutchinson} to propose the
 delayed logistic equation (also known as Hutchinson's equation)
 \begin{equation}
 v'(t)=\alpha v(t)\left[1-v(t-\tau)\right],
 \end{equation}
where we have normalized the carrying capacity to unity (by setting $v(t)=N(t)/K$), and introduced the time delay $\tau>0$.
Interestingly, independent of Hutchinson, in 1955 Wright \cite{wright}, motivated by an unpublished note of Lord Cherwell about a heuristic approach to the density of prime numbers, studied in detail the delay differential equation
	 \begin{equation}
	u'(t)=-\alpha u(t-\tau)\left[1+u(t)\right].
	 \end{equation}
Wright's equation is an equivalent form of the delayed logistic equation, established via the change of variables $u(t)=v(t)-1$.
	
Since the delayed logistic equation exhibits stable periodic solutions whenever $\alpha \tau >\pi/2$, it is very tempting to use this modification of the logistic ordinary differential equation to explain observed biological oscillations. One example is the work of May fitting the delayed logistic equation to the empirical data of oscillatory blowfly populations from Nicholson's experiments \cite{may}. However, the use of Hutchinson's delayed logistic equation received criticisms from biological modellers due to the lack of a mechanistic biological derivation \cite{GeritzKisdi}, the main objection being that it is not based on clearly defined birth and death terms, and the delay was inserted on a purely phenomenological basis. Indeed, 
in May's work the parameters in the mathematical model did not directly correspond to any biological parameters. 

Later, Gurney et al.~\cite{gurney} proposed the equation 
	 \begin{equation}
	 N'(t)=aN(t-\tau)e^{-bN(t-\tau)}-\mu N(t),
	  \end{equation}
which is known today as Nicholson's blowfly equation. It is structurally different from the delayed logistic equation, and the terms have clear biological interpretation as the birth and mortality rates, while the delay represents the period between egg laying and hatching.
Furthermore, in contrast to the delayed logistic equation, Nicholson's blowfly equation can explain not only the appearance of cycling populations, but also the two bursts of reproductive activity per adult population cycle that appeared in Nicholson's data.

Incorporating a delay in the growth term, Arino et al. \cite{arino} derived an alternative formulation of the delayed logistic equation
	 \begin{equation}
	N'(t)=\frac{\gamma\mu N(t-\tau)}{\mu e^{\mu\tau}+\kappa(e^{\mu\tau}-1) N(t-\tau)}-\mu N(t)-\kappa N(t)^2,
	 \end{equation}
which has recently been extended to distributed delays \cite{lin}. This model behaves similarly to the classical logistic equation, in the sense that it cannot sustain periodic oscillations, however large delays cause the extinction of the population. In another recent work, Lindstr\"om studied chemostat models with time lag in the conversion of the substrate into biomass \cite{lindstrom}, and obtained equations of logistic type with delays as a limiting case. Lindstr\"om's equation shares the same dynamical properties as the alternative delayed logistic equation from \cite{arino}, both generating a monotone semiflow, unlike Hutchinson's equation.

In addition to these biological applications, the delayed logistic equation motivated the development of a large number of analytical and topological tools, including local and global Hopf-bifurcation analysis for delay differential equations \cite{chow,faria,hassard,nussbaum}, asymptotic analysis \cite{fowler}, 3/2-type stability criteria \cite{ivanov,wright}, and the study of slowly oscillatory solutions \cite{lessard}. The most famous related problem is Wright's conjecture from 1955, which proposes global asymptotic stability in the delayed logistic equation for $\alpha \tau \leq \pi/2$. This long-standing question has been resolved recently by combining analytical and computational tools \cite{krisztin,vBJ}. Moreover, many variants and generalizations of the delayed logistic equation have been studied in the mathematical literature, for instance 
equation including positive instantaneous feedback \cite{gynr}, non-autonomous terms \cite{faria}, neutral terms \cite{gopalsamy}, spatial diffusion \cite{zou}, and multiple delays \cite{yan}. Further examples are discussed in \cite{kuang,liz,ruan}. 

In this paper we derive a novel type of delayed logistic equation, inspired by cell biology. Our equation includes both terms with and without delays, as well as with distributed delays. After presenting derivation of the model in Section 2, we study its basic mathematical properties in Section 3, such as well-posedness, existence and stability of equilibria. The global dynamics is fully described in Section 4, and we provide a complete description of the global attractor. While all non-trivial solutions converge to the positive equilibrium, the model exhibits long-lasting transient oscillations of various shapes; this phenomenon is discussed in Section 5. Finally, we summarize our findings and discuss their implications for future research in Section 6.

%====================================================================

\section{Derivation of the new delayed logistic equation}

We derive our new delayed logistic equation by considering the mean-field limit of a simplistic agent-based, on-lattice model that describes both cell motility and proliferation, and cell-cell interactions through the incorporation of volume exclusion. More specifically, we assume that agents move and proliferate on an $n$-dimensional square lattice with spacing $\Delta$ and length (in each direction) $\ell\Delta$, where $\ell$ is an integer describing the number of lattice sites.   A great interest in current medical research is to understand the roles of two particular phenotypes, namely ``high proliferation-low migration'' and ``low proliferation-high migration'', in the progression of agressive cancers \cite{levchenko}. Our model is designed to capture the ``go or grow'' hypothesis frequently acknowledged in the cancer literature \cite{farin,giese}, which proposes that cell proliferation and migration are temporally exclusive events. 

As such, we divide our agent population into two subpopulations, motile and proliferative. Each agent is assigned to a lattice site, from which it can move or proliferate into an adjacent site. If a motile agent attempts to move into a site that is already occupied, the movement event is aborted. Similarly, if a proliferative agent attempts to proliferate into a site that is already occupied, the proliferation event is aborted. Agents convert from being motile to proliferative at constant rate $r>0$ per unit time. That is, $r\delta{t}$ is the probability that a motile agent switches to become proliferative in the next infinitesimally small time interval $\delta{t}$. Upon switching, agents remain immobile for a fixed time $\tau>0$ (representing the length of time taken for cells to progress through the cell cycle to division), and they then attempt to proliferate by placing a daughter agent into one of the nearest neighbour lattice sites. After attempting proliferation, proliferative agents switch back to being motile. The initial agent distribution is, on average, spatially uniform, and is achieved by populating lattice sites uniformly at random with probability $P_s$. 

Following arguments presented in Baker and Simpson~\cite{Baker:2010:CMF}, and closing the hierarchy of moment equations at the pair level (assuming that the occupancy of neighbouring lattice sites is independent), we can derive a system of delay differential equations that describe how the density of agents on the lattice evolves over time:
\begin{eqnarray}
\label{equation:m}
m'(t)&=&-rm(t)+rm(t-\tau)+rm(t-\tau)\left(\frac{K-p(t)-m(t)}{K}\right),\\
p'(t)&=&rm(t)-rm(t-\tau),
\label{equation:p}
\end{eqnarray}
where $m$ is the density of motile agents, $p$ is the density of proliferative agents, $K>0$ is the number of sites on the lattice (one can think of the carrying capacity as $K=\ell^n$), and $'$ denotes differentiation with respect to $t$. The equations encode the model assumption that motile agents become proliferative with rate $r$, and stay in this proliferative phase for time $\tau$. Once the cell cycle is completed, proliferative cells switch back to being motile cells again. As they switch, they attempt to place a daughter agent into a lattice site; the probability that site is empty (in the mean-field limit) is $\left(K-p(t)-m(t)\right)/K$.

System~\eqref{equation:m}--\eqref{equation:p} can be rewritten as a scalar equation using the natural biological relation 
\begin{equation}
p(t)=\int_{t-\tau}^{t}rm(s)\,\text{d}s,
\end{equation}
that is, the proliferative cells at time $t$ are exactly those who entered the proliferative subpopulation in the time interval $[t-\tau,t]$. Using this relation in equation~\eqref{equation:m} gives
\begin{equation}
m'(t)=-rm(t)+rm(t-\tau)+rm(t-\tau)\left(1-\frac{r}{K}\int_{t-\tau}^t m(s) \,\text{d}s-\frac{m(t)}{K}\right).
\end{equation}
Next we rescale the density with $K$ by letting $\tilde{m}(t)=m(t)/K$,
\begin{equation}
\tilde{m}'(t)=-r \tilde m(t)+r \tilde m(t-\tau)+r\tilde m(t-\tau)\left(1-r \int_{t-\tau}^t \tilde m(s) \,\text{d}s-\tilde m(t)\right),
\end{equation}
and drop the tildes and rearrange to
\begin{equation}
m'(t)=-r m(t)+r m(t-\tau)\left(2-r \int_{t-\tau}^t m(s)\,\text{d}s-m(t)\right). 
\label{eq:resc1}
\end{equation}
 To complete the non-dimensionalisation, we also rescale time by letting $\tilde{t}=t/\tau$ and $x(\tilde{t})=m(t)$.
Then, ${\text{d}x}/{\text{d}\tilde t}=\tau m'(t)$ and, moreover, $m(t-\tau)=m(\tau \tilde t-\tau)=m(\tau(\tilde t-1))=x(\tilde t -1)$. Noting that
\begin{equation}
\int_{t-\tau}^t m(s)\,\text{d}s=\tau \int_{\tilde t-1}^{\tilde t }x(\tilde{s}) \,\text{d}\tilde s,
\end{equation}
we can, again, drop the tildes, to arrive at the new delayed logistic equation:
\begin{equation}
\frac{\text{d}x}{\text{d}t}=- r \tau x(t)+r\tau x( t-1)\left(2-r\tau  \int_{t-1}^t x( s) \,\text{d}s-x(t)\right). 
\label{eq:resc2}
\end{equation}

%====================================================================

\section{Basic mathematical properties} 

With the notation $\rho=r\tau>0$, we now analyse the scalar differential equation with discrete and distributed delays
\begin{equation} 
x'(t)=-\rho x(t)+\rho x(t-1)\left(2-\rho \int_{t-1}^t x(s) \,\text{d}s-x(t)\right). 
\label{eq}
\end{equation}
Let $C=C([-1,0],R)$ denote the Banach space of continuous real-valued functions on the interval $[-1,0]$ equipped with the supremum norm. Then equation \eqref{eq} is of the form $x'(t)=f(x_t)$, where the solution segment $x_t \in C$ is defined by 
\begin{equation} 
x_t(s):=x(t+s), \quad s \in [-1,0]. 
\end{equation}
Then the map $f: C \to R$  is given by
\begin{equation} 
f(\phi)=-\rho\phi(0)+\rho \phi(-1)\left(2-\rho \int_{-1}^0 \phi(s) \,\text{d}s-\phi(0) \right), \label{f} 
\end{equation}
and the initial data is specified by
\begin{equation} 
x_0=\phi \in C. \label{iv} 
\end{equation}

%====================================================================

\subsection{Well-posedness}
By a solution of equations~\eqref{eq},\eqref{iv} we mean a continuous function $x(t)$ on an interval $[-1,A)$ with $0<A\leq \infty$, which is differentiable on $(0,A)$, satisfies equation~\eqref{eq} on $(0,A)$ and also satisfies equation~\eqref{iv}.

\begin{lem} For every $\phi \in C$, there exists a unique solution of the initial value problem~\eqref{eq},\eqref{iv} defined on an interval $[-1,A)$ for some $0<A\leq \infty$, that depends continuously on the initial data.
\end{lem}

\begin{proof}
Recall the standard existence and uniqueness theorem for functional differential equations \cite{smith}.  We shall verify that the local Lipschitz property holds for $f$, \textit{i.e.} for any $M>0$, there is an $L>0$ such that
\begin{equation}
|f(\phi)-f(\psi)| \leq L ||\phi-\psi||, \text{ whenever } ||\phi||<M,\,\,||\psi||<M.
\end{equation}
For a constant $M>0$, let $||\phi||<M$ and $||\psi||<M$, then we have the estimates
\begin{eqnarray}
\nonumber
|f(\phi)-f(\psi)|&\leq& 
\rho ||\phi-\psi||+2 \rho ||\phi-\psi|| \vphantom{\int_{-1}^0}\\
\nonumber
&&+\rho \left\vert \phi(-1)\left(\rho \int_{-1}^0 \phi(s) \,\text{d}s+\phi(0) \right) - \psi(-1)\left(\rho \int_{-1}^0 \psi(s) \,\text{d}s+\psi(0) \right) \right\vert \\
\nonumber
&\leq& 3\rho  ||\phi-\psi|| \vphantom{\int_{-1}^0}\\
\nonumber
&&+\rho\left\vert \phi(-1)\left(\rho \int_{-1}^0 \phi(s) \,\text{d}s+\phi(0) \right) - \psi(-1)\left(\rho \int_{-1}^0 \phi(s) \,\text{d}s+\phi(0) \right)\right\vert\\
\nonumber
&&+\rho \left\vert \psi(-1)\left(\rho \int_{-1}^0 \phi(s) \,\text{d}s+\phi(0) \right)  - \psi(-1)\left(\rho \int_{-1}^0 \psi(s) \,\text{d}s+\psi(0) \right)\right\vert \\
&\leq& \left(3\rho +\rho^2M+\rho M+M \rho^2+\rho M \right)||\phi-\psi|| \vphantom{\int_{-1}^0}. 
\end{eqnarray}
Hence $L=3\rho +M\left(2\rho^2 +2\rho\right)$ is a valid Lipschitz constant. Existence, uniqueness and continuous dependence then follow from the general theory, c.f. Theorem 3.7 of \cite{smith} and \cite[Chapter 2]{kuang}. \qed

\end{proof}
Due to biological constraints, we are interested only in non-negative solutions, \textit{i.e.} $x_0(s)\geq 0$, $s \in [-1,0]$. More specifically, we consider solutions with $x_0 \in \Omega$, where
\begin{equation}
\Omega:=\left\{\phi \in C: \phi(s)\geq 0 \text{ for } s\in [-1,0] \text{ and }  \phi(0)+\rho \int_{-1}^0 \phi(s) \,\text{d}s \leq1\right\}.
\end{equation}
This set in fact corresponds to the biologically feasible phase space since, for a solution with $x_t \in \Omega$,
\begin{equation}
\theta(t):=x(t)+\rho \int_{t-1}^t x(s) \,\text{d}s
\end{equation} 
represents the total cell density (accounting for all mobile and proliferating cells) that should, in the non-dimensional model, lie between zero and unity. We shall use $\Omega \subset C$, the set of biologically feasible states as our phase space, noting that $\Omega$ depends on the parameter $\rho$.

%====================================================================

\subsection{Positivity and boundedness from above}

\begin{lem}
The set $\Omega$ is positively invariant.
\end{lem}

\begin{proof}
 We claim that if $\theta(0)\leq 1$ then $\theta(t)\leq 1$ for $t>0$. Note that 
\begin{equation}
\theta'(t)=-\rho x(t)+\rho x(t-1)\left(2-\rho \int_{t-1}^t x(s) \,\text{d}s-x(t)\right)+\rho x(t)-\rho x(t-1),
\end{equation}
which simplifies to
\begin{equation}
\theta'(t)=\rho x(t-1)(1-\theta(t)). \label{eq:theta}
\end{equation}
For $w(t):=1-\theta(t)$ we have $w'(t)=-\theta'(t)=-\rho x(t-1) w(t)$, hence 
\begin{equation}
w(t)=w(0)\exp\left(-\rho \int_0^t x(s-1)\,\text{d}s\right)\geq 0,
\label{w}
\end{equation} 
whenever $w(0)\geq 0$, or equivalently $\theta(0)\leq 1$. We find that 
\begin{equation}
\theta(t)=x(t)+\rho \int_{t-1}^t x(s) \,\text{d}s\leq 1,
\end{equation} 
holds if 
\begin{equation}
x(0)+\rho \int_{-1}^0 x(s) \,\text{d}s\leq 1.
\end{equation}
The general positivity condition, \textit{i.e.} that $x_t\geq 0$ whenever $x_0\geq 0$, for system \eqref{eq},\eqref{iv} is that $x_t \geq 0$ and $x(t)=0$ implies $f(x_t) \geq 0$ \cite[Theorem 3.4]{smith}. This clearly holds as for $x(t)=0$ we have $f(x_t)=\rho x(t-1)(2-\theta(t))\geq 0$ following from $x_t\geq 0$ and $\theta(t) \leq 1$. Since $x_t \in \Omega$ is equivalent to $x_t \geq 0$ and $\theta(t)\leq 1$, we obtain that $\Omega$ is positively invariant, and additionally the estimate $x(t)<1$ holds. 
\qed
\end{proof}

Consequently, solutions starting from $\phi \in \Omega$ exist globally, \textit{i.e.} on the interval $[-1,\infty)$, with $x^\phi_t \in \Omega$ for all $t\geq 0$.

%====================================================================

\subsection{Steady states and their stability}

\begin{thm}
Equation \eqref{eq} has two steady states, $x_0=0$ which is always unstable, and $x_*=1/(\rho+1)$ which is always locally asymptotically stable.
\end{thm}

\begin{proof} 
Setting $x'=0$, we obtain the steady state equation
\begin{equation}
0=-\rho x+\rho x \left(2-\rho x -x \right),
\end{equation}
which has the solutions $x_0=0$, and $x_*=1/(\rho+1).$

Linearisation of equation~\eqref{eq} around $x=0$ gives
\begin{equation}
x'(t)=-\rho x(t)+2\rho x(t-1),
\end{equation}
with the characteristic equation
\begin{equation}
\lambda+\rho =2\rho  e^{-\lambda},
\end{equation}
which always has a positive real root, hence $0$ is unstable.

Next we look at the stability of the equilibrium $x_*$. 
Using the definition in equation~\eqref{f}, we have
\begin{eqnarray}
\nonumber
f(x_*+\phi)&=&-\rho (\phi(0)+x_*)+\rho (\phi(-1)+x_*)\left(2-\rho \int_{-1}^0 (\phi(s)+x_*) \,\text{d}s-(\phi(0)+x_*)\right)\\&=&L(\phi)+g(\phi),
\end{eqnarray}
which has linear part
\begin{equation}
L(\phi)=-\rho \left( 1+\frac{1}{\rho+1}\right)\phi(0)+\rho \phi(-1)-  \frac{\rho^2}{\rho+1} \int_{-1}^0 \phi(s) \,\text{d}s,
\end{equation}
and 
\begin{equation}
g(\phi)=\rho \phi(-1)\left(-\rho \int_{-1}^0 \phi(s) \,\text{d}s-\phi(0)\right).
\end{equation}
Clearly $\lim_{||\phi|| \to 0} {|g(\phi)|}/{||\phi||}=0$, hence the linear variational equation is
\begin{equation}
y'(t)=-\rho \left( 1+\frac{1}{\rho+1}\right)y(t)+\rho y(t-1)-  \frac{\rho^2}{\rho+1} \int_{t-1}^t y(s) \,\text{d}s.
\end{equation}
The exponential ansatz $y(t)=e^{\lambda t}$ gives the characteristic equation
\begin{equation}
\lambda=- \rho \left( 1+\frac{1}{\rho+1}\right)- \frac{\rho^2}{\rho+1}  \int_{t-1}^t e^{\lambda (s-t)} \,\text{d}s +\rho e^{-\lambda}, \label{charint}
\end{equation}
which becomes, after integration and rearranging, 
 \begin{equation}
\lambda^2=-\rho \left( 1+\frac{1}{\rho+1}\right)\lambda - \frac{\rho^2}{\rho+1}  (1-e^{-\lambda }) +\rho \lambda e^{-\lambda },
 \end{equation}
provided $\lambda \neq 0$. Note that $\lambda=0$ is not a root of equation~\eqref{charint}, otherwise
\begin{equation}
0=-\rho \left( 1+\frac{1}{\rho+1}\right) - \frac{\rho^2}{\rho+1}  +\rho =-\rho <0,
\end{equation}
which is a contradiction. Therefore we may write the characteristic equation as
\begin{equation}
\chi(\lambda)=P(\lambda)+Q(\lambda)e^{-\lambda }=0,
\label{charpol}
\end{equation}
with 
\begin{eqnarray}
P(\lambda)&=&\lambda^2+\rho \left( 1+\frac{1}{\rho+1}\right)\lambda+\frac{\rho^2}{\rho+1},\\ 
Q(\lambda)&=&-\frac{\rho^2}{\rho+1} -\rho\lambda.
\end{eqnarray} 
We can factor the characteristic equation as
\begin{equation}\chi(\lambda)=\left(\lambda+\frac{\rho}{\rho+1}\right)\left(\lambda+\rho-\rho e^{-\lambda}\right)=0, \label{chareq}
\end{equation}
which demonstrates that we always have the real root $\lambda=-{\rho}/({\rho+1})<0$. The other roots are the roots of $\lambda+\rho-\rho e^{-\lambda}=0$. This is a well-known type of characteristic equation of the form $\lambda=A+Be^{-\tau \lambda}$, see for example Chapter 4.5 of \cite{smith}. However, we are on the stability boundary of this equation given that $\lambda=0$ is a root (which is, as we established, not a root of the original characteristic equation \eqref{charint}). We can quickly show that all other roots have imaginary parts less than zero: assuming a root with $\Re \lambda \geq 0$ but $\lambda \neq 0$, we have
\begin{equation}
\rho < \vert \lambda+\rho\vert = \vert \rho e^{-\lambda}\vert \leq \rho,
\end{equation}
which is a contradiction. Therefore the equilibrium $x_*$ is locally asymptotically stable. \qed
\end{proof}

%====================================================================

\section{Global dynamics}

%====================================================================

\subsection{Persistence} 

\begin{thm} \label{persistence}
Equation \eqref{eq} is strongly uniformly persistent, that is there exists a $\delta>0$, independent of the initial data, such that for any solution $x^\phi(t)$ with $\phi(0)>0$, 
\begin{equation}
\liminf_{t \to \infty} x(t) \geq \delta.
\end{equation}
\end{thm}

\begin{proof}

Let 
\begin{equation}
g(\psi) =\rho \psi(-1)\left(2-\psi(0)-\rho\int_{-1}^0 \psi(\theta)\,\text{d}\theta\right),
\end{equation}
then equation \eqref{eq} can be written as
\begin{equation}
x'(t)=-\rho x(t)+g(x_t). 
\end{equation}
From the variation of constants formula, for any $\sigma>\omega>0$ we find
\begin{equation}
x(\sigma)=e^{-\rho(\sigma-\omega)}\left(x(\omega)+\int_{\omega}^{\sigma} e^{\rho(s-\omega)} g(x_s) \,\text{d}s\right). 
\end{equation}
We use the notation $x_\infty=\liminf_{t \to \infty} x(t)$, and $x^\phi$ to denote a specific solution with initial function $\phi$. Assume the contrary of the statement of the theorem. Then there is a sequence $\phi_n$, such that $\lim_{n \to \infty} x_\infty^{\phi_n} =0$.  Let $q \in (0,1)$ such that $q^4>1/2$ (any $q>0.841$ is just fine). Then there is a $T_n \to \infty$, such that $x^{\phi_n}(t) \in [qx_\infty^{\phi_n},1]$ for all $t>T_n-1$. 
Furthermore, there is a $t_n>T_n+n$ such that $x^{\phi_n}(t_n) \in  [qx_\infty^{\phi_n},q^{-1}x_\infty^{\phi_n}] $.
For the particular case $x = x^{\phi_n}, \sigma = t_n$
and $\omega = T_n$, the variation of constants formula gives
\begin{equation}
x^{\phi_n}(t_n)=e^{-\rho(t_n-T_n)}\left(x^{\phi_n}(T_n)+\int_{T_n}^{t_n} e^{\rho(s-T_n)} g\left(x^{\phi_n}_s\right) \,\text{d}s\right). 
\end{equation}
Using the integral mean value theorem, there is an $\eta_n \in [T_n,t_n]$, such that
\begin{equation}
e^{-\rho t_n}\int_{T_n}^{t_n} e^{\rho s} g\left(x^{\phi_n}_s\right)\,\text{d}s =e^{-\rho t_n}g\left(x^{\phi_n}_{\eta_n}\right)\int_{T_n}^{t_n} e^{\rho s} \,\text{d}s= g\left(x^{\phi_n}_{\eta_n}\right){\rho^{-1}}\left(1-e^{\rho(T_n-t_n)}\right). 
\end{equation}
Then we have the relation
\begin{equation}
	x^{\phi_n}(t_n)=
	e^{-\rho(t_n-T_n)} x^{\phi_n}\left(T_n\right)+g\left(x^{\phi_n}_{\eta_n}\right){\rho^{-1}}\left(1-e^{\rho(T_n-t_n)}\right). \label{perz}
\end{equation}
Since $x^{\phi_n}(t_n)\to 0$,  $e^{-\rho(t_n-T_n)} x^{\phi_n}(T_n)\to 0$ and $\left(1-e^{T_n-t_n}\right) \to 1$ as $n \to \infty$,
necessarily $g\left(x^{\phi_n}_{\eta_n}\right) \to 0$ as well as $n \to \infty$. Since $x^{\phi_n}_{\eta_n} \in \Omega,$ we have
$g\left(x^{\phi_n}_{\eta_n}\right) \geq \rho x^{\phi_n}_{\eta_n}(-1) $, hence $x^{\phi_n}(\eta_n-1) \to 0$ as well as $n \to \infty$. 

Next we claim that 
\begin{equation}
\left(2-x^{\phi_n}(\eta_n) -\rho\int_{-1}^0 x^{\phi_n}(\eta_n+\theta)\,\text{d}\theta\right) \to 2 \text{ as } n \to \infty.
\end{equation}
From the inequality $x'(t)\geq - \rho x(t)$, we find that $x^{\phi_n}(\eta_n+\theta) \leq e^{\rho} x^{\phi_n}(\eta_n-1)$ for $\theta \in [-1,0]$. Hence, 
\begin{equation}
x^{\phi_n}(\eta_n) + \rho\int_{-1}^0 x^{\phi_n}(\eta_n+\theta)\,\text{d}\theta \leq e^{\rho} x^{\phi_n}(\eta_n-1)(1+\rho),
\end{equation}
and the right-hand side is already shown to converge to zero.
For sufficiently large $n$, we have the estimate $g\left(x^{\phi_n}_{\eta_n}\right) \geq 2q \rho x^{\phi_n}_{\eta_n}(-1) \geq 2q \rho q x^{\phi_n}_{\infty}$ and $\left(1-e^{\rho(T_n-t_n)}\right)>q $, therefore from equation~\eqref{perz} we obtain
\begin{equation}
q^{-1} x^{\phi_n}_{\infty} \geq  x^{\phi_n}(t_n)\geq 2q^{3} x^{\phi_n}_{\infty}.
\end{equation}
We find $1\geq 2 q^4$, which contradicts the choice of $q$.
\qed
\end{proof}

%====================================================================

\subsection{A crucial convergence property}

\begin{thm} For all positive solutions,  we have
\begin{equation}
\lim_{t \to \infty} \left(x(t)+\rho \int_{t-1}^t x(s) \,\mathrm{d}s \right)=1.
\end{equation}
\end{thm}

\begin{proof}
This follows from the standard comparison principle, since for large $t$ we have $1>x(t-1)>{\delta}/{2}>0$, and by equation~\eqref{eq:theta}, we see that $\theta(t)$ is squeezed between $\underline \theta(t)$ and $\bar \theta(t)$,  which are the solutions of
\begin{equation}
\underline \theta'(t)=\rho  \frac{\delta}{2} (1-\underline \theta(t)),
\end{equation} 
and
\begin{equation}
\bar \theta'(t)=\rho (1-\bar \theta(t)),
\end{equation} 
with $\underline \theta(0)=\bar \theta(0)=\theta(0)>0$, each converging to unity. \qed
\end{proof}

%====================================================================

\subsection{Global asymptotic stability} 

\begin{thm} All positive solutions of equation~\eqref{eq} in $\Omega$ converge to the positive equilibrium. \label{attractivity}
\end{thm}

\begin{proof} First we recall the following theorem from \cite{gyoripituk}:
\begin{thmx}[Gy\H ori-Pituk]\label{gyp} Consider 
\begin{equation}
x'(t)=(a_0+a(t))x(t)+(b_0+b(t))x(t-\tau),
\end{equation}
where $a_0,b_0\in R$, $\tau>0$ are constants and $a,b: [0,\infty) \to R$ are continuous functions. Let the following assumptions be satisfied:
\begin{enumerate}[label=(\roman*),leftmargin=1cm]
\item $\lambda=a_0+b_0 e^{-\lambda \tau}$ has a unique root $\lambda_0$ with largest real part, and this root $\lambda_0$ is real and simple;
\item $a(t) \to 0$, $b(t)\to 0$ as $t \to \infty$;
\item $\int_0^\infty a^2(t)\,\mathrm{d}t<\infty$, $\int_0^\infty b^2(t)\,\mathrm{d}t<\infty$;
\item $\int_0^\infty \left|\tau a(t)-\int_{t-\tau}^t a(s)\,\mathrm{d}s\right| \mathrm{d}t<\infty$, $\int_0^\infty \left|\tau b(t)-\int_{t-\tau}^t b(s)\,\mathrm{d}s\right|\mathrm{d}t<\infty$.
\end{enumerate}
Then, every solution $x(t)$ satisfies 
\begin{equation}
x(t)=\exp\left(\int_0^t \lambda(s)\,\mathrm{d}s\right)\left(\xi+o(1)\right), \quad t \to \infty,
\end{equation}
where 
\begin{equation}
\lambda(t)=\lambda_0+\left(1+b_0\tau e^{-\lambda_0 \tau}\right)^{-1}\left(a(t)+e^{-\lambda_0\tau} b(t)\right),
\end{equation} 
and $\xi=\xi(x)$ is a constant depending on the solution $x$.
\end{thmx}

\noindent
Fix some $\psi \in \Omega$, and let $x^\psi(t)$ be the corresponding solution with 
\begin{equation} 
w^\psi(t)=1-x^\psi(t)-\rho \int_{t-\tau}^t x^\psi(s)\,\text{d}s=w(0)\exp\left(-\rho\int_0^t x^\psi(s-1)\,\text{d}s\right),
\label{wpsi} 
\end{equation}
where the last equality was given in equation~\eqref{w}. Now consider the equation 
\begin{equation} 
z'(t)=-\rho z(t)+\rho z(t-1)\left(1+w^\psi(t)\right). 
\label{eqpsi}
\end{equation}
We apply Theorem \ref{gyp}, and check all the conditions, where we have $a_0=-\rho$, $b_0=\rho$, $\tau=1$, $a(t)=0$ and $b(t)=\rho w^\psi(t)$. The conditions with $a(t)$ are trivial, and the conditions with $b(t)$ follow from the combination of persistence and equation~\eqref{wpsi} as follows, since $w$ is exponentially convergent.

\begin{enumerate}[label={\it(\roman*)},leftmargin=1cm]
\item $\lambda=-\rho+\rho e^{-\lambda }$ has a unique root $\lambda_0$ with largest real part, and this root $\lambda_0$ is real and simple (in fact $\lambda_0=0$).
\item To see that $w^\psi(t)\to 0$ as $t \to \infty$, we can use Theorem 4.1: there is a $\delta>0$ and a $T$ such that $x(t)>\delta$ for $t>T-1$. From equation~\eqref{wpsi} we find
\begin{eqnarray} 
\nonumber
w^\psi(t)&=&w(0)\exp\left(-\rho\int_0^T x^\psi(s-1)\,\text{d}s\right)\exp\left(-\rho\int_T^t x^\psi(s-1)\,\text{d}s\right)\\
&<& w(0)\exp\left(-\rho \int_0^T x^\psi(s-1)\,\text{d}s\right)e^{-\rho \delta (t-T)},
\label{wexp} 
\end{eqnarray}
or
\begin{equation} 
w^\psi(t) < Ke^{-\rho\delta t}, \quad K=w(0)\exp\left(-\rho\int_0^T x^\psi(s-1)\,\text{d}s\right)e^{r\delta T}.
\label{wexp2} 
\end{equation}
\item The estimate in equation~\eqref{wexp2} shows that 
\begin{equation}
\int_0^\infty \left(w^\psi(t)\right)^2 \,\text{d}t<K ^2\int_0^\infty e^{-2\rho\delta t} \,\text{d}t = \frac{K^2}{2\rho\delta} < \infty.
\end{equation}
\item From the estimate in equation~\eqref{wexp2}, we also find that
\begin{equation}
\int_0^\infty \left|w^\psi(t)-\int_{t-1}^t w^\psi(s)\,\text{d}s\right|\text{d}t<\int_0^\infty  w^\psi(t)\,\text{d}t+ \int_0^\infty \int_{t-1}^t w^\psi(s)\,\text{d}s\,\text{d}t<\infty.
\end{equation}
\end{enumerate}
As a result, 
\begin{equation}
\lambda(t)=(1+\rho)^{-1}\rho w^\psi(t),
\end{equation}
and all the conditions of Theorem 4.3 hold. Thus, every solution $z(t)$ satisfies 
\begin{equation}
z(t)= \exp\left((1+\rho)^{-1}\rho\int_0^t  w^\psi(s)\,\text{d}s\right)(\xi+o(1)), \quad t \to \infty.
\end{equation}
Now notice that a solution $x(t)$ of equation~\eqref{eq} with initial function $\psi$ is also a solution of equation~\eqref{eqpsi}, hence it converges to a constant. In view of Theorem \ref{persistence}, this constant can only be the positive equilibrium. 

\end{proof}

%====================================================================

\subsection{The global attractor} 

\begin{thm} The global attractor $\mathcal{A}$ consists of the two equilibria and a heteroclinic orbit connecting the two.
\end{thm}

\begin{proof} First we demonstrate that the global attractor $\mathcal{A}$ of $\Phi^\Omega$ exists. From the invariance of the bounded set $\Omega$, this semiflow generated by the equation on $\Omega$ is point dissipative. It follows from the Arzel\`a-Ascoli theorem that the solution operators $\Phi^\Omega_t$ are completely continuous for $t \geq 1$, then by applying Theorem 3.4.8 of \cite{hale88}, the compact global attractor exists. 

The two equilibria are part of the global attractor. Next we show that there exists a unique heteroclinic orbit in $\Omega$ connecting $x_0=0$ and the positive equilibrium, $x_*$. Recall that the characteristic equation of the linearization at $x_0=0$ is $\lambda+\rho=2\rho e^{-\lambda} ,$ which has  a single positive root $\lambda_0$. The other roots
form a sequence of complex conjugate pairs $(\lambda_j,\bar
\lambda_j)$ with $\Re \lambda_{j+1} < \Re \lambda_j < \lambda_0$ for
all integers $j \geq 1$. We refer to Chapter XI of \cite{dvvw} for a complete analysis
of the characteristic equation of the form $z-\alpha-\beta e^{-z}$, in particular Fig. XI.1. in \cite{dvvw} which summarizes its properties. Our equation is a special case with $\alpha=-\rho$, $\beta=2\rho$.

There is a $\gamma>0$ such that $\lambda_0
> \gamma > \Re \lambda_j$ for all $j \geq 1$, and there is a $k \in \N$ such
that $\Re \lambda_{k} > 0$ and $\Re \lambda_{k+1} \leq 0$, where the point $(-\rho,2\rho)$ lies between the curves $C_k^-$ and $C_{k+1}^-$ on the $(\alpha,\beta)$-plane, given by 
\begin{equation}
C_j^-=\left\{(\alpha,\beta)=\left( \frac{\nu \cos\nu}{\sin \nu}, -\frac{\nu}{\sin\nu}\right) : \, \nu \in \left((2j-1)\pi, 2j \pi \right)\right\}.
\end{equation}
The intersections of $C_j^-$ and the half-line $(-\rho,2\rho)$ ($\rho>0$) are determined by $\cos \nu = 1/2$, thus $\nu=2j\pi-\pi/3$,
$\sin \nu = -\sqrt{3}/2$ and the intersection is given by $\rho=\rho_j=(2j\pi-\pi/3)/\sqrt{3}.$ Hence, $k$ is either zero  or the largest positive integer such that $\rho>\rho_k$.

Considering the leading real root, the corresponding eigenfunction is
given by  $\chi_0(s):=e^{\lambda_0 s}$, $s \in [-1,0].$ The phase
space $C$ can be decomposed as $C=P_0 \oplus P_k \oplus Q$, where the
function $\chi_0 \in C$ spans the linear eigenspace $P_0:=\{c \chi_0
\, : \, c \in \R \}$. $P_k$ is a $2k$-dimensional eigenspace corresponding to
$\{\lambda_j:j=1,\ldots,k\}$. If $\lambda_0$ is the only eigenvalue with positive real
part then $P_k=\emptyset$. There is an $\varepsilon>0$ such that $\Re \lambda_{k} > \varepsilon$. $Q$ corresponds to the remaining part of the spectrum.

There exist open neighborhoods $N_0$, $N_k$, $M$ in
$P_0$, $P_k$, $Q$, respectively, and  $C^1$-maps $w_0: N_0 \to P_k\oplus Q$,  
$w_u: N_0\oplus N_k \to Q$ with range
in $N_k\oplus M$ and $M$, respectively, such that  $w_0(0)=0$, $Dw_0(0)=0$,
$w_u(0)=0$, $Dw_u(0)=0$; where $D$ denotes the Fr\'echet-derivative. Then, the $\gamma$-unstable set
of the equilibrium $x_0=0$, namely
\begin{eqnarray}
\nonumber 
\mathcal{W}_0(0)&:=&\left\{ \vphantom{e^{\gamma t}} \phi \in N_0+N_k+M : \hbox{there is a trajectory } z_t, \, t \in \R \hbox{ with } z_0= \phi, \right. \\ && \left. \quad\quad z_t \in N_0+N_1 \hbox{ when } t \leq 0 \hbox{
	and } z(t)e^{-\gamma t} \to 0 \hbox { as } t \to -\infty\right\},
\end{eqnarray}
coincides with the graph
\begin{equation}
\mathcal{W}_0:=\{ \phi +w_0(\phi) : \phi \in N_0 \}.
\end{equation}
Furthermore, the unstable set
of the equilibrium $x_0=0$ 
\begin{eqnarray}
\nonumber
\mathcal{W}_u(0)&:=&\left\{ \vphantom{e^{-\varepsilon t}} \phi \in N_0+N_k+M : \hbox{there is a trajectory } z_t, t \in \R \hbox{
	with } z_0= \phi, \right.\\ && \left. \quad\quad z_t \in N_0+N_1 \hbox{when } t \leq 0 \hbox{
	and } z(t)e^{-\varepsilon t} \to 0 \hbox { as } t \to -\infty\right\},
\end{eqnarray} 
coincides with the graph
\begin{equation}
\mathcal{W}_u:=\{ \phi +w_u(\phi) : \phi \in N_0+N_k \}.
\end{equation}
For the details see \cite{fhwmanifold,kww}.

For any $\phi \in N_0$, there is a $c \in \R$ such that
$\phi=c\chi_0$. We have $||\chi_0||=1$ and $ \chi_0(t) \geq
e^{-\lambda_0} > 0$ for all $ t \in [-1,0]$. It follows from
$Dw(0)=0$ that 
\begin{equation}
\lim_{||\phi|| \to 0} \frac{||w(\phi)||}{||\phi||}
= 0, \, \, \phi \in N_0,
\end{equation} 
which means
\begin{equation} 
\lim_{c
	\to 0} \frac{||w(c \chi_0)||}{|c|} = 0.
	\end{equation}  
Therefore, there exists a
$c_0>0$ such that $||w(c\chi_0)||/|c|< e^{-\lambda}/2$ whenever $c \in (0,c_0)$, or equivalently $||w(c \chi_0)
|| < ce^{-\lambda}/2$. We may assume that $c_0<1/(2(1+\rho))$. Then
\begin{equation}
\min \left\{ c\chi_0(t)+w(c\chi_0)(t) : t \in [-1,0], \,\, c \in (0,c_0) \right\} \geq
ce^{-\lambda}- \frac{c}{2} e^{-\lambda} = \frac{c}{2}
e^{-\lambda } > 0,
\end{equation} 
and 
\begin{eqnarray}
\nonumber
c\chi_0(0)+w(c\chi_0)(0)+ r \int_{-1}^0 \left\{c\chi_0(s)+w(c\chi_0)(s)\right\}\,\text{d}s &\leq& (1+\rho)(c+||w(c\chi_0)||)\\
\nonumber
&<&2c_0(1+\rho)\\
&<&1,
\end{eqnarray}
thus
\begin{equation}
\phi_c:= c\chi_0+w(c\chi_0) \in
\mathcal W_0 \cap \Omega \hbox{ for all } c \in (0,c_0).
\end{equation}

The unstable
set $\mathcal{W}_0(0)$ intersects $\Omega$, and for any
function $\phi_c$ of this intersection, there is a complete solution
$x(t):\R \to \R^+_0$ such that $x_0=\phi_c$ and $x_t \to 0_*$ as $t
\to -\infty$. By the negative invariance of $\mathcal W_0(0)$, the
trajectory is completely in $\Omega$, hence we have proved the existence of a heteroclinic orbit connecting the two equilibria.

For uniqueness, consider the unstable manifold $\mathcal{W}_u(0)$.
Any solution on the set $\mathcal{W}_u(0)\backslash \mathcal{W}_0(0)$ oscillates about $0$ as
$t \to - \infty$, hence can not be a positive heteroclinic solution. Denote the unique heteroclinic orbit in $\Omega$ by $\mathcal{H}$. We claim that $\mathcal{H}$ and the two equilibria constitute the global attractor, otherwise there are further complete orbits in $\Omega$. Assume there is a $\psi \in \Omega$ which is on such a complete orbit $\Gamma \subset \Omega$, so  $\psi \neq 0$, $\psi \neq x_*$, and $\psi \notin \mathcal{H}$. 
Then its $\alpha$-limit set $\alpha(\psi)$ and $\omega$-limit set $\omega(\psi)$ are non-empty, compact and invariant. In particular, by Theorem \ref{attractivity}, $\omega(\psi)=\{x_*\}$. If $\alpha{(\psi)}=\{x_0\}$, then $\Gamma$ is a connecting orbit coinciding with $\mathcal{H}$. Hence, we may assume that there is a $\phi \neq 0$ such that $\phi \in \alpha(\psi)$. Then $x_t^\phi \to x_*$ as $t \to \infty$, so $x_* \in \alpha(\psi)$ as well. But this contradicts the stability of $x_*$. Thus, the global attractor does not contain any other orbit, and $\mathcal{A}=\{x_0\} \cup\{x_*\}\cup \mathcal{H}$.
 
\end{proof}

%====================================================================

\section{Numerical simulations and metastability}

%====================================================================

\begin{figure}
	\begin{center}
		\includegraphics[scale=0.31]{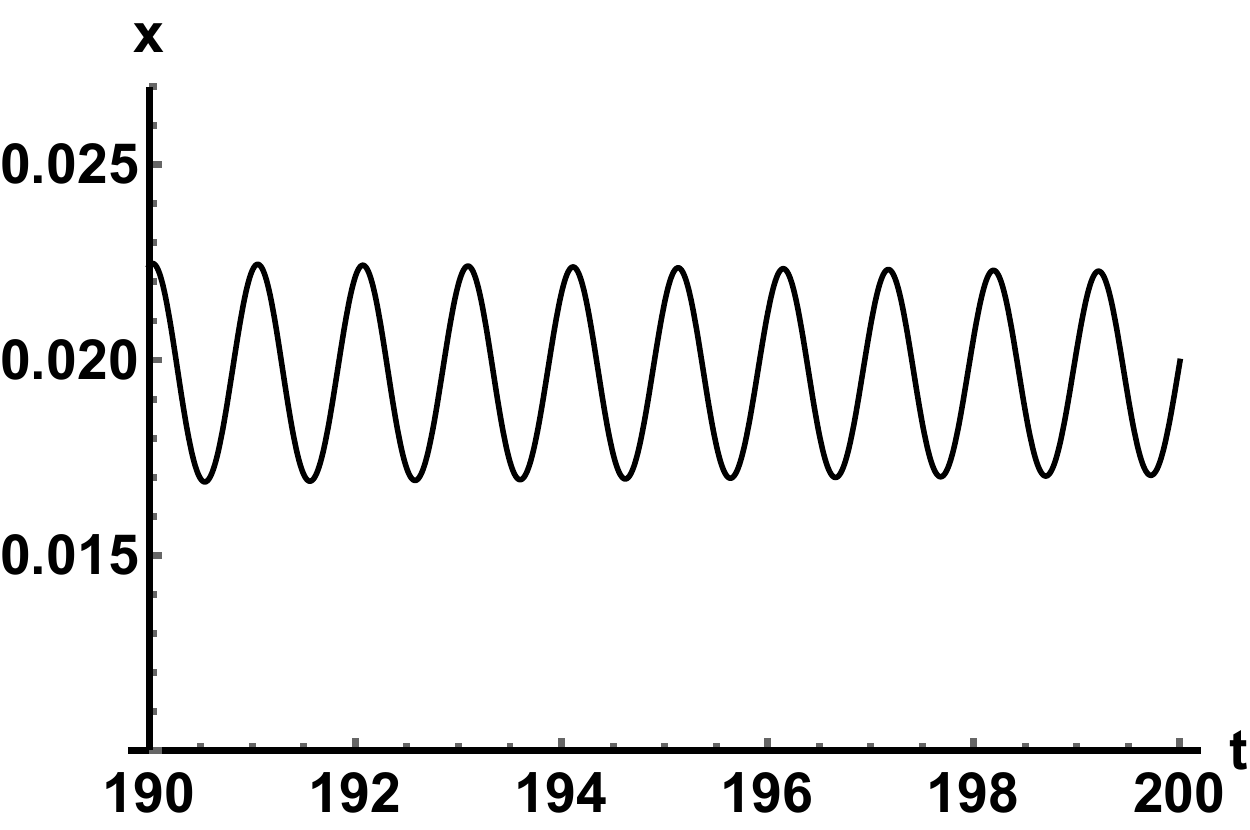}
		\includegraphics[scale=0.31]{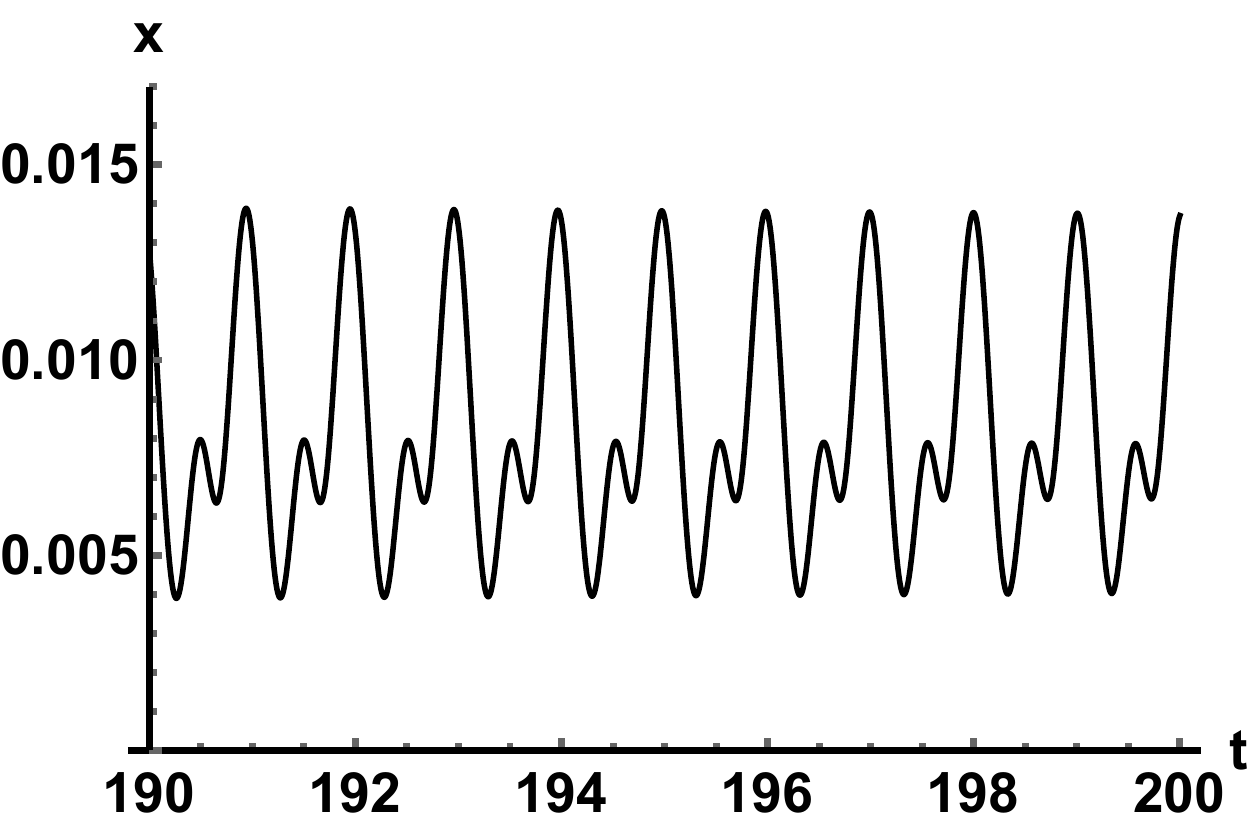} 
		\includegraphics[scale=0.31]{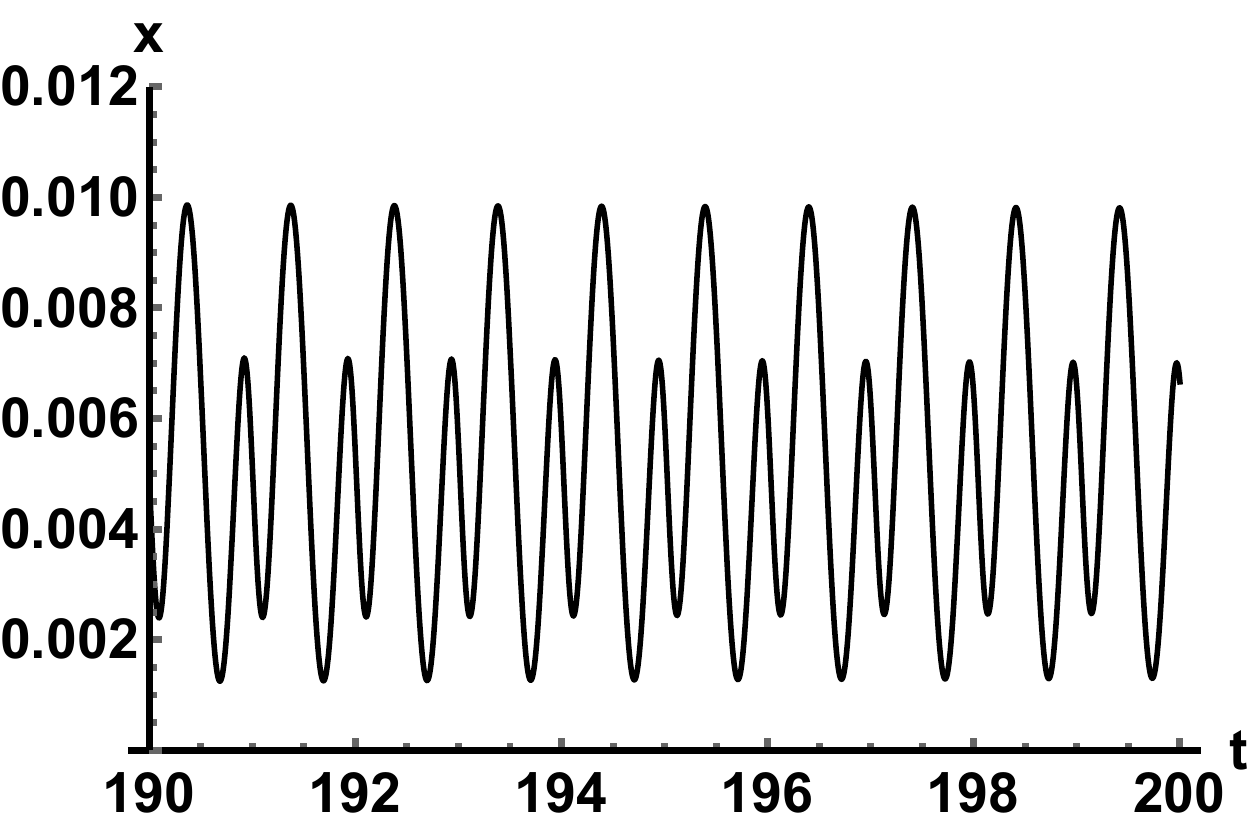}
	\end{center}
	\caption{Representative solutions that appear periodic over short time scales. From left to right, $\rho=50,120,190$ with $\phi(t)=0.005(\cos(10t)+1)$.}
	\label{figure:solutions}
\end{figure}

Representative numerical simulations of solutions of equation~\eqref{eq} are plotted in {Fig.~\ref{figure:solutions}, with initial function $\phi(t)=0.005(\cos(10t)+1)$ and parameter values $\rho=50,120,190$. As we can see, in a short time interval of length 10, the solutions appear to be periodic, with their periods approximately equal to the delay, despite the fact that we proved in Theorem \ref{attractivity} that all non-trivial solutions converge to the positive equilibrium. Indeed, viewing the solution on a larger time scale in Fig. \ref{figure:transient}, we can see that the solution is in fact not periodic, however the convergence is very slow. This is a situation that is sometimes called metastability. Paraphrasing \cite{holmes}, metastability refers to a situation when something appears not to change on short time scales, while it changes after a sufficiently long period of time. Such phenomena are common in boundary value problems of partial differential equations, and metastability has been observed for delay differential equations as well, see, for example,~\cite{erneux,grotta}. Long lasting transient oscillations have been reported in \cite{morozov} for a scalar differential equation with a constant delay. It is interesting to see that equation \eqref{eq} can exhibit rapid convergence for small $\rho$ and oscillatory patterns of various shapes for larger $\rho$, as illustrated in Fig.~\ref{figure:convergence} where solutions are plotted with different initial functions.

%====================================================================

\begin{figure}
	\begin{center}
		\includegraphics[scale=0.44]{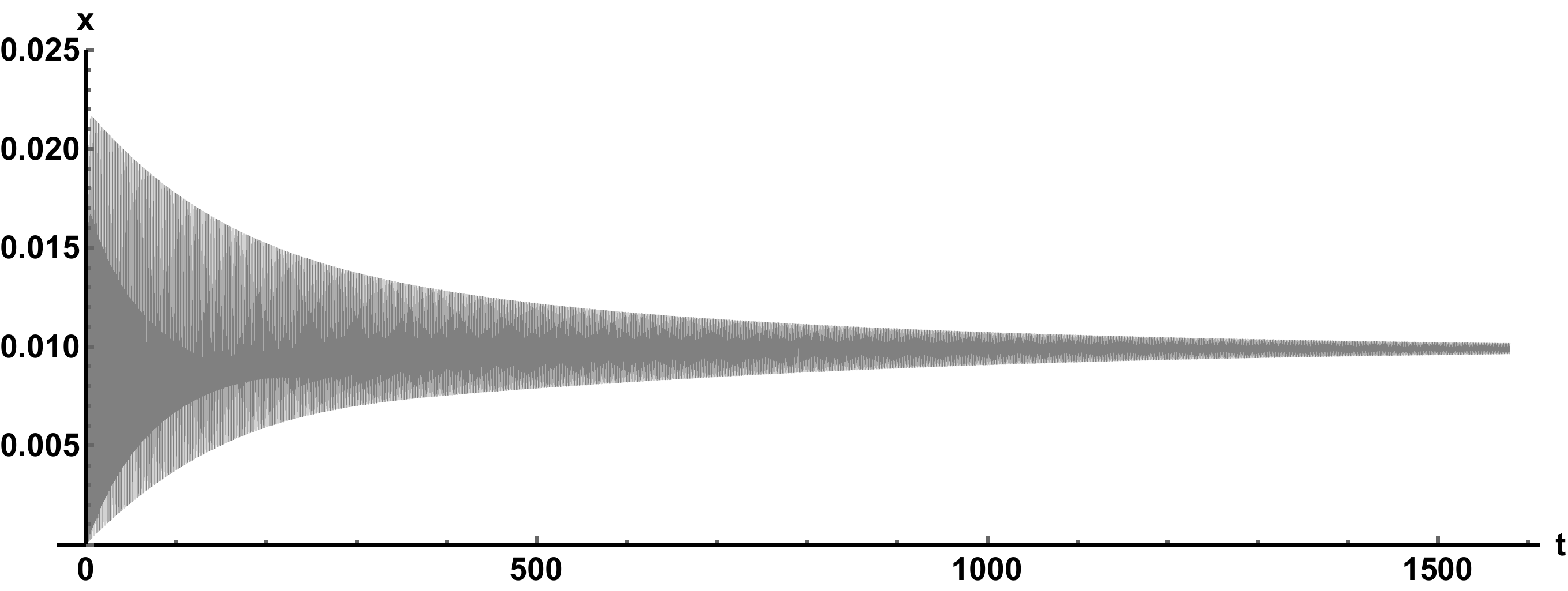} \\
		\includegraphics[scale=0.22]{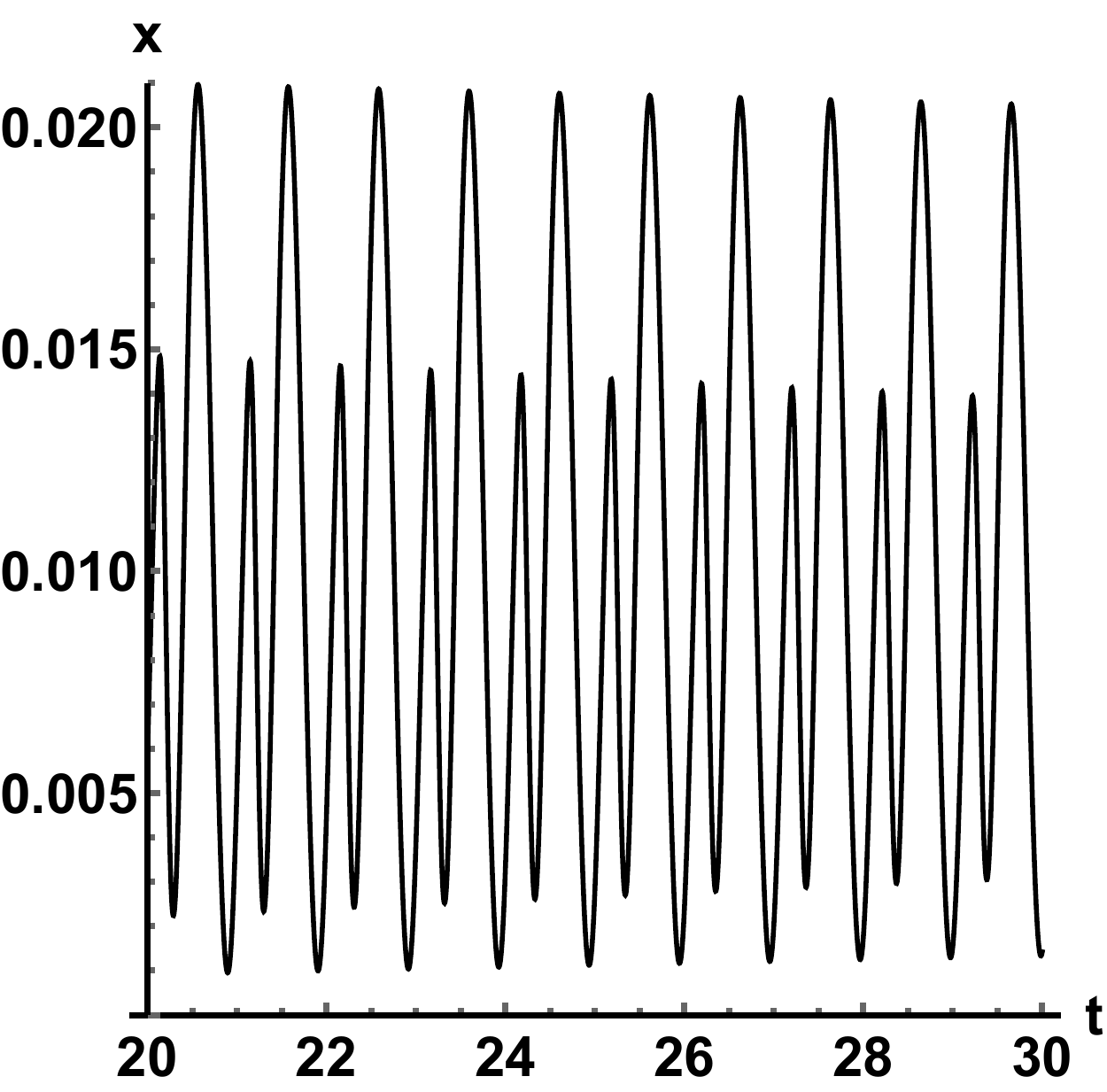} \,
		\includegraphics[scale=0.22]{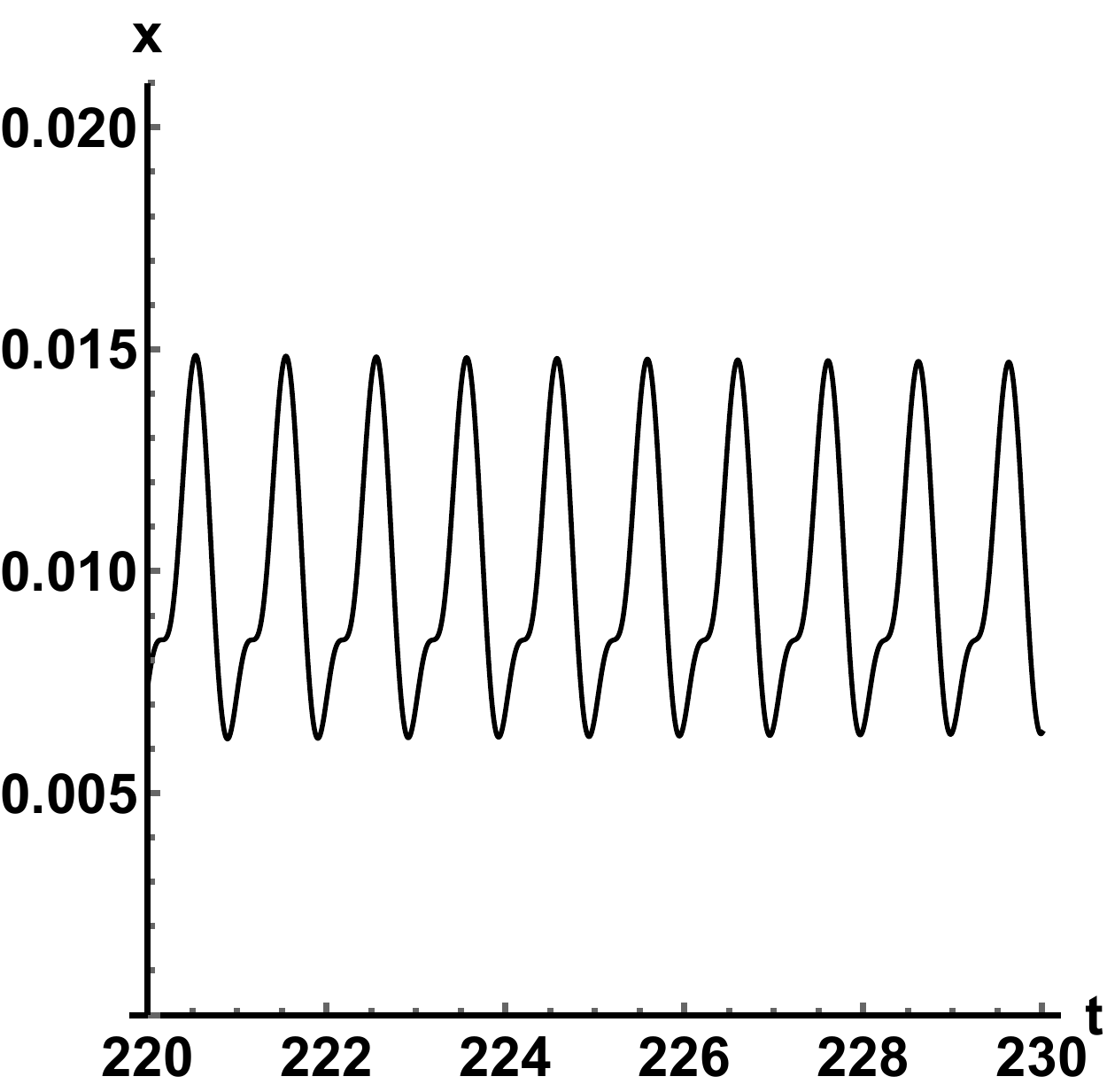} \, \includegraphics[scale=0.22]{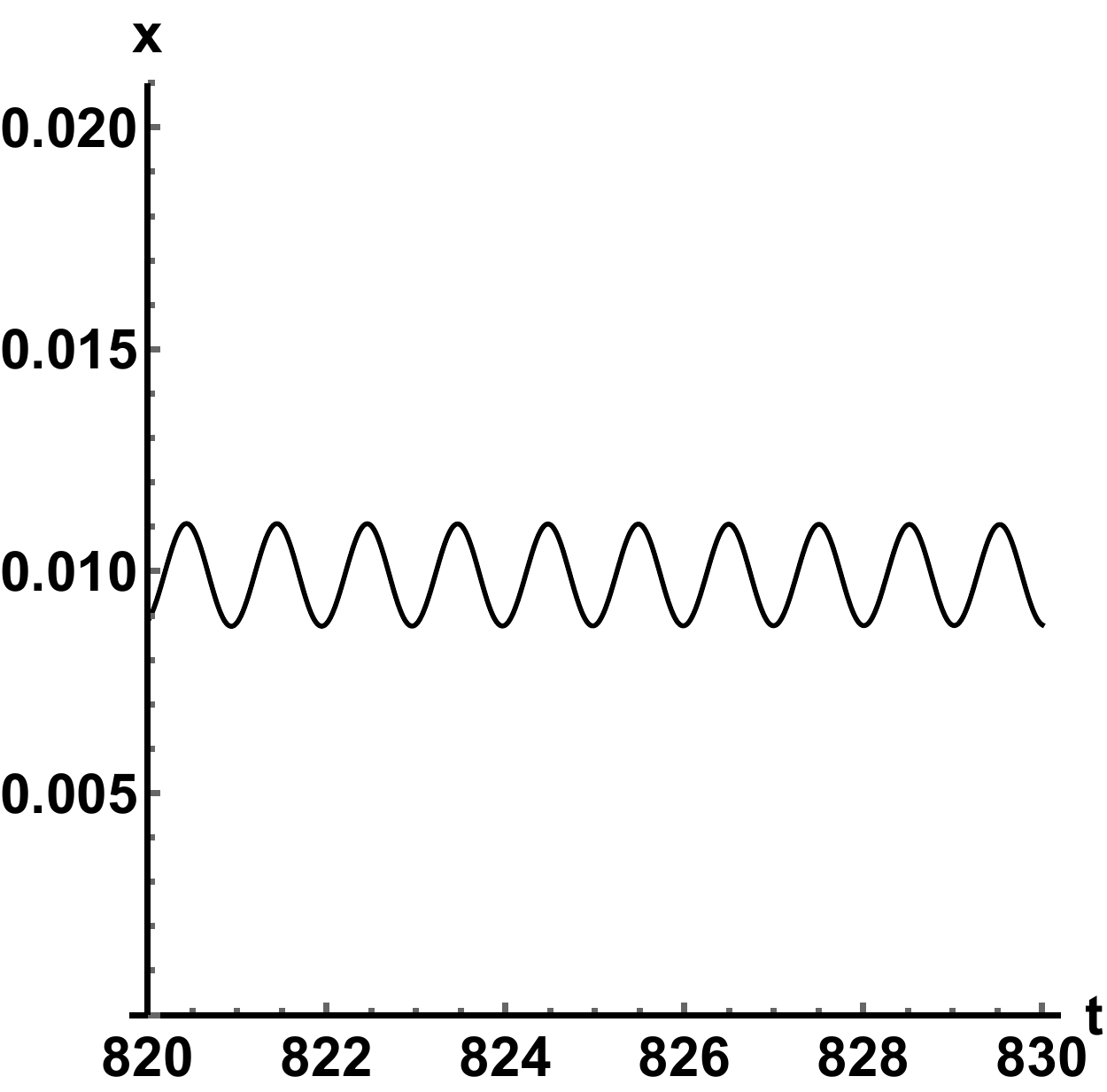} \,\includegraphics[scale=0.22]{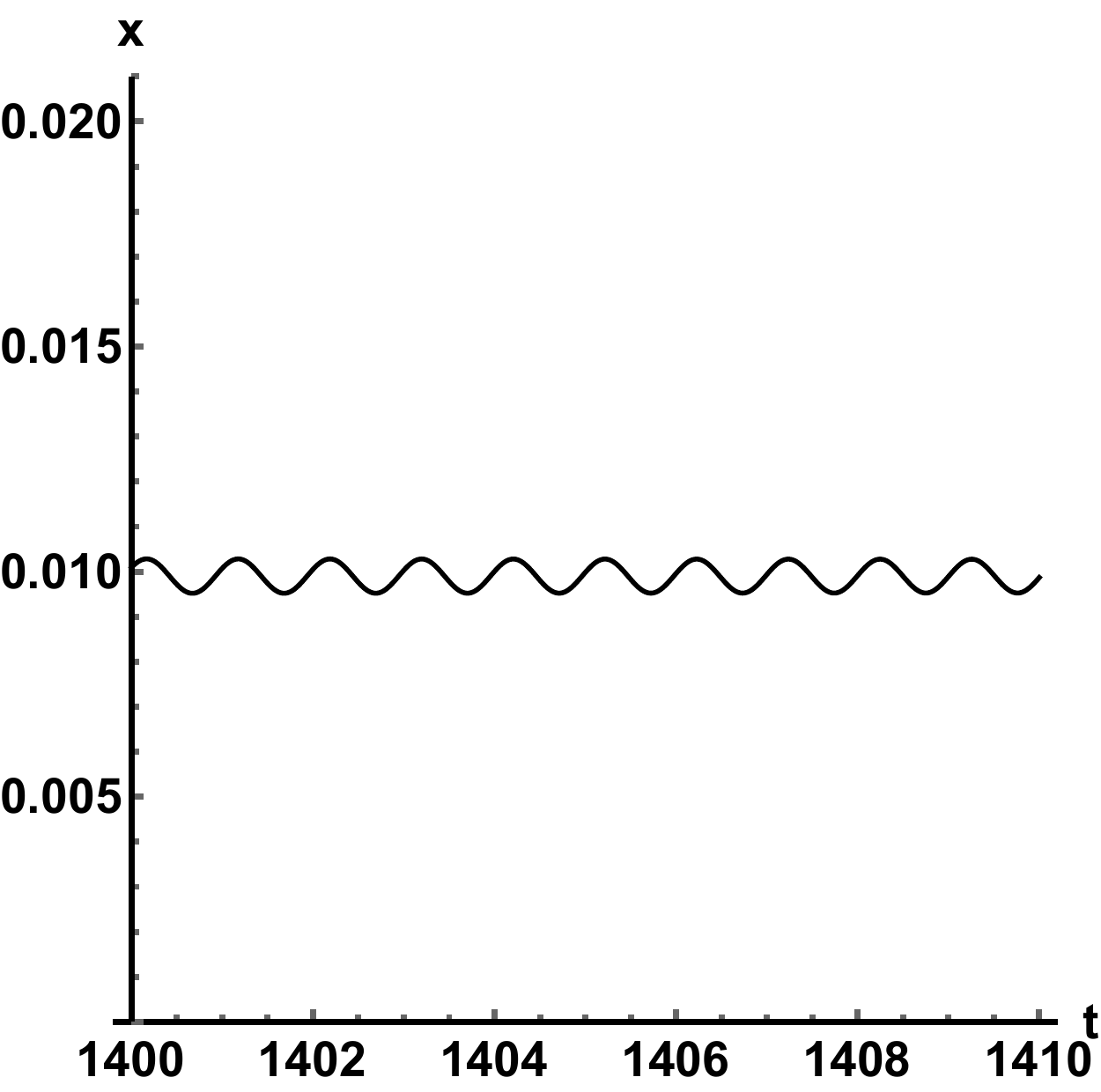}
	\end{center}
	\caption{Top: envelope of the solution with initial function $\phi(t)=0.005(\cos(10t)+1)$ and $\rho=100$, showing slow convergence to the positive equilbrium. Bottom: snapshots of the same solution over different time intervals. }
	\label{figure:transient}
\end{figure}

%====================================================================

\begin{figure}
\begin{center}
	\includegraphics[scale=0.45]{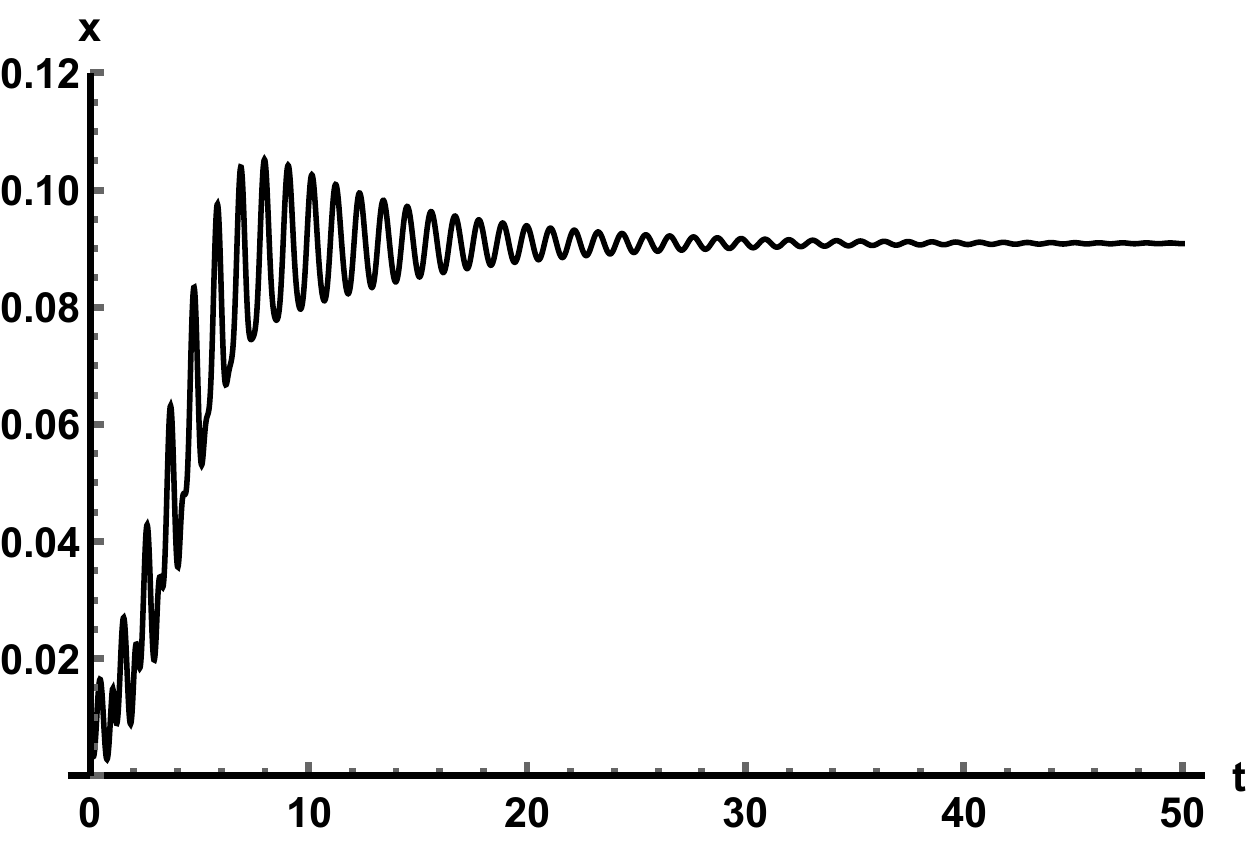} \quad 
	\includegraphics[scale=0.45]{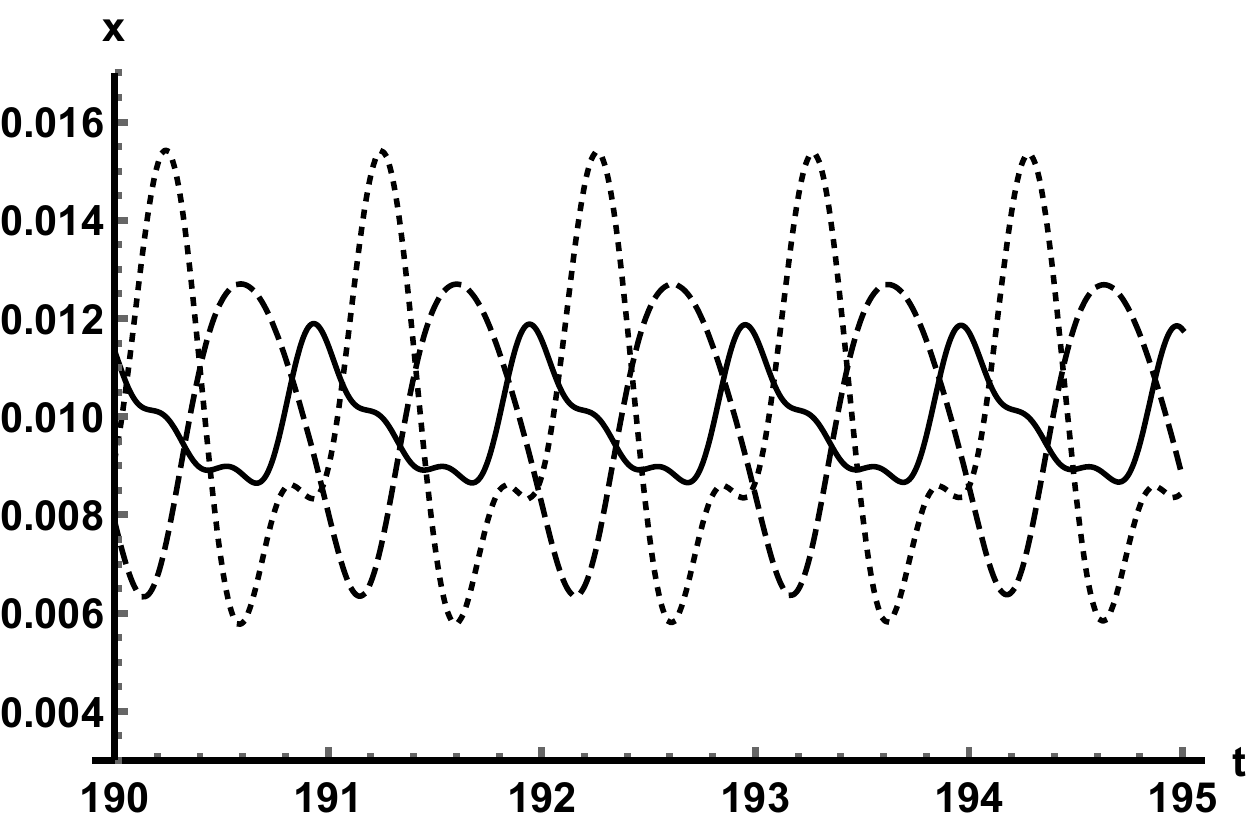}
	\end{center}
	\caption{Left: Convergence for $\rho=10$ with initial function $\phi(t)=0.005(\cos(10t)+1)$. Right: Various periodic patterns for $\rho=100$, emerging from initial functions $\phi(t)=0.005(\cos(10t)+1)$ (dotted), $\phi(t)=0.005(\cos(20t)+1)$ (solid), $\phi(t)=0.005(\cos(10t^4)+1)$ (dashed).}
	\label{figure:convergence}
\end{figure}

%====================================================================

To gain some understanding of this phenomenon, we look at our model from the point of view of singular perturbations. In equation~\eqref{eq}, let us divide by $\rho^2$. For large $\rho$, we use the notation $\varepsilon=1/\rho$, then we have 	 	
\begin{equation} 
\varepsilon^2 x'(t)=\varepsilon\left[-x(t) + 2x(t-1)-x(t-1)x(t) \right]  - x(t-1)\int_{t-1}^t x(s) \,\text{d}s. \label{epseq}
\end{equation}
Setting $\varepsilon=0$, we obtain the singular perturbation
\begin{equation}
0=x(t-1)\int_{t-1}^t x(s)\,\text{d}s.
\end{equation}
Besides $x\equiv 0$, any periodic function with unit period that is on average zero is a solution of this equation, which may be a hint as to why we see long-lasting transient oscillations with a variety of shapes.

In addition, we can look at the singular perturbation of the characteristic equation: equation~\eqref{chareq} can be written as
\begin{equation} \lambda^2+\rho \left( 1+\frac{1}{\rho+1}\right)\lambda+\frac{\rho^2}{\rho+1}=\left(\frac{\rho^2}{\rho+1} +\rho\lambda\right)e^{-\lambda }. \end{equation}
Multiplying by $(\rho+1)/\rho^2$ we have
\begin{equation} 
\lambda^2\frac{\rho+1}{\rho^2}+ \frac{\rho+2}{\rho}\lambda+1=\left(1 +\lambda\frac{\rho+1}{\rho}\right)e^{-\lambda }, 
\end{equation} 
which is
\begin{equation} \lambda^2(\varepsilon^2+\varepsilon)+ (1+2\varepsilon)\lambda+1=\left(1 +\lambda (1+2\varepsilon)\right)e^{-\lambda }, \end{equation} 
with the notation $\varepsilon=1/\rho$. Setting $\varepsilon=0$, we obtain the singular perturbation
\begin{equation}  
\lambda+1=\left(1 +\lambda\right)e^{-\lambda }. 
\end{equation} 
The roots of this equation satisfy either $\lambda=-1$ or $e^{-\lambda }=1$, having the purely imaginary roots $i=2k\pi$, $k\in \Z$. Indeed, as we increase $\rho$, we can see that all complex roots line up along the imaginary axis, see Fig.~\ref{spectrumfig}.

\begin{figure}
\begin{center}
	\includegraphics[scale=0.31]{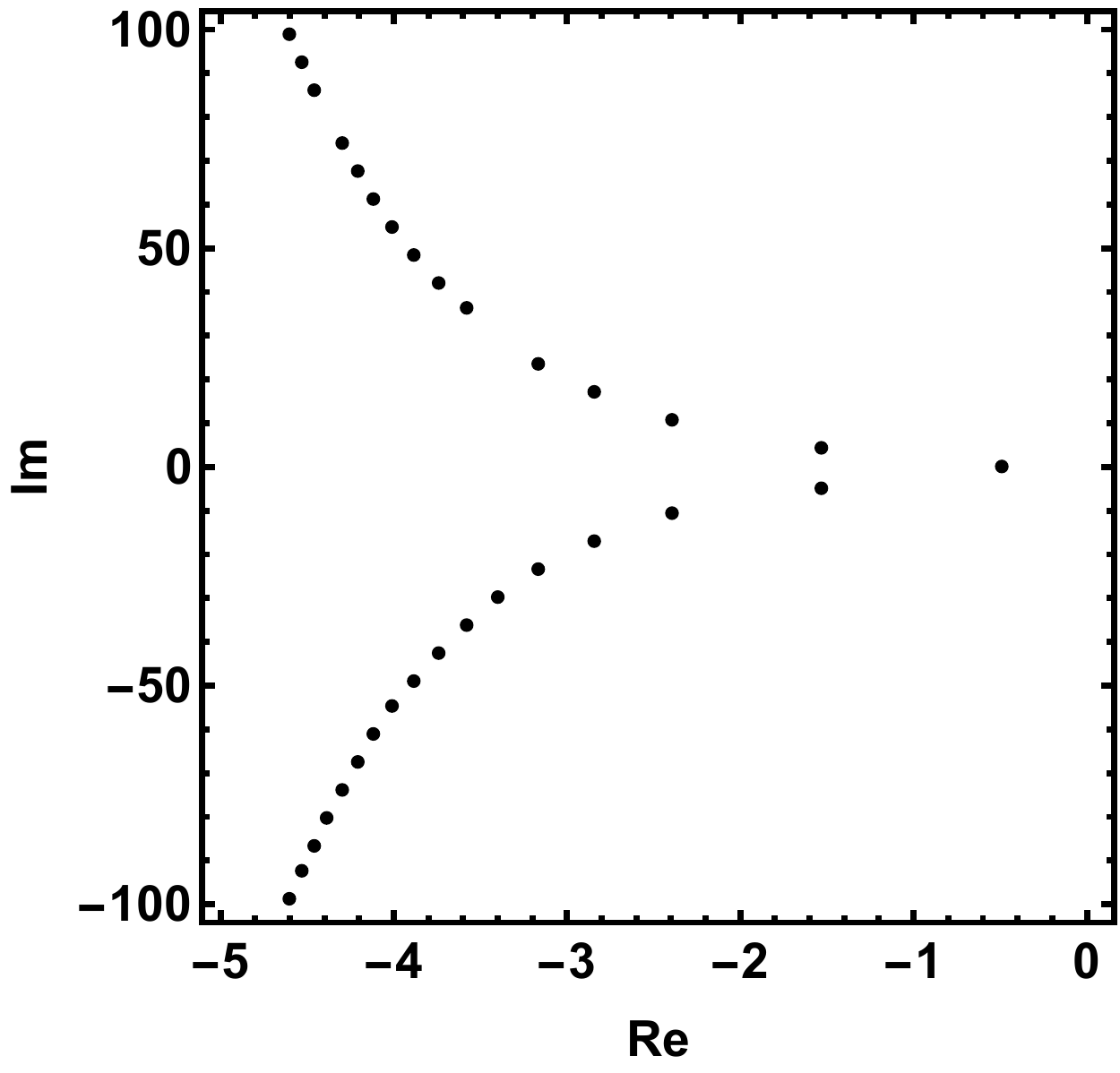} 
	\includegraphics[scale=0.31]{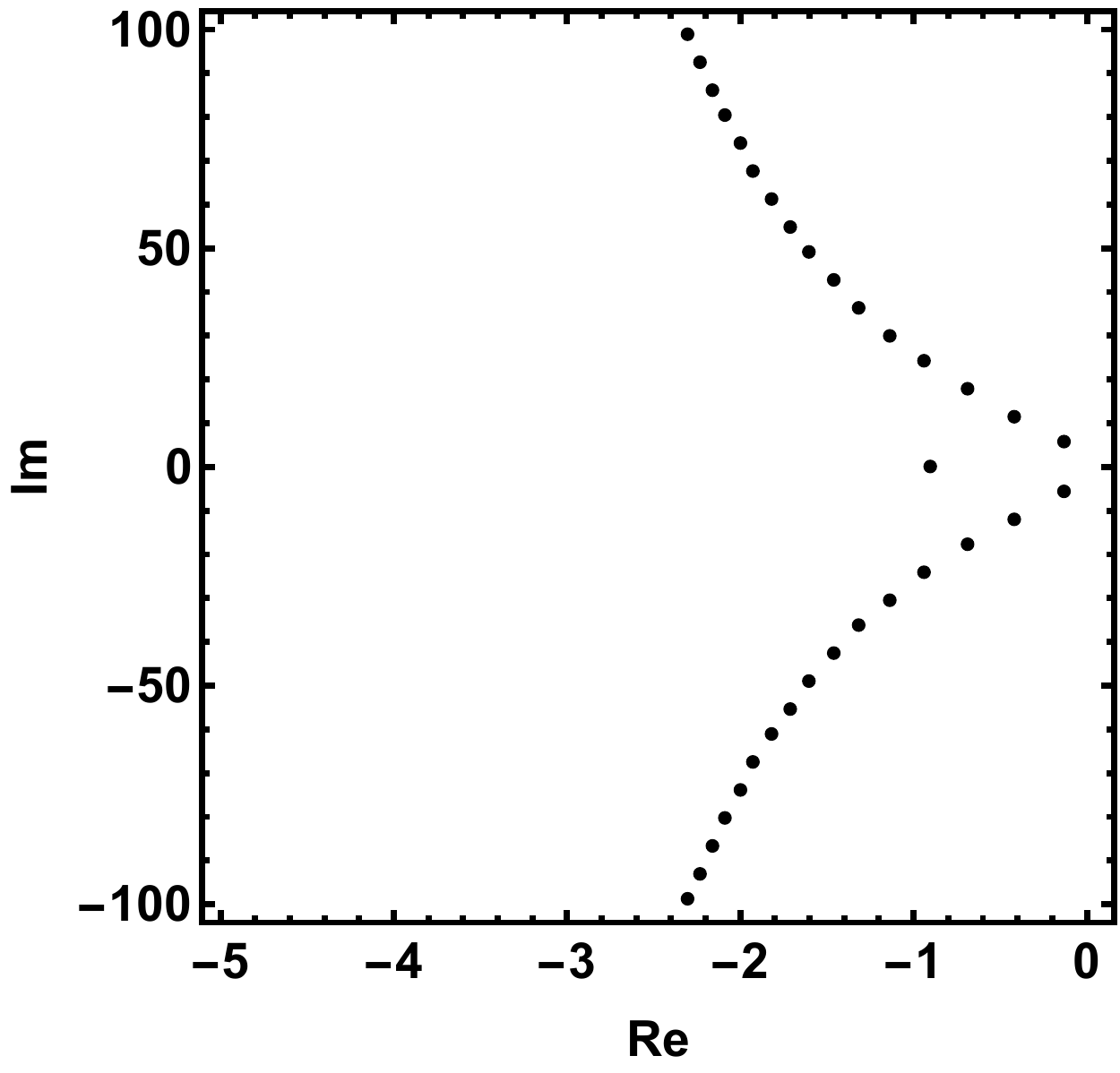}
	\includegraphics[scale=0.31]{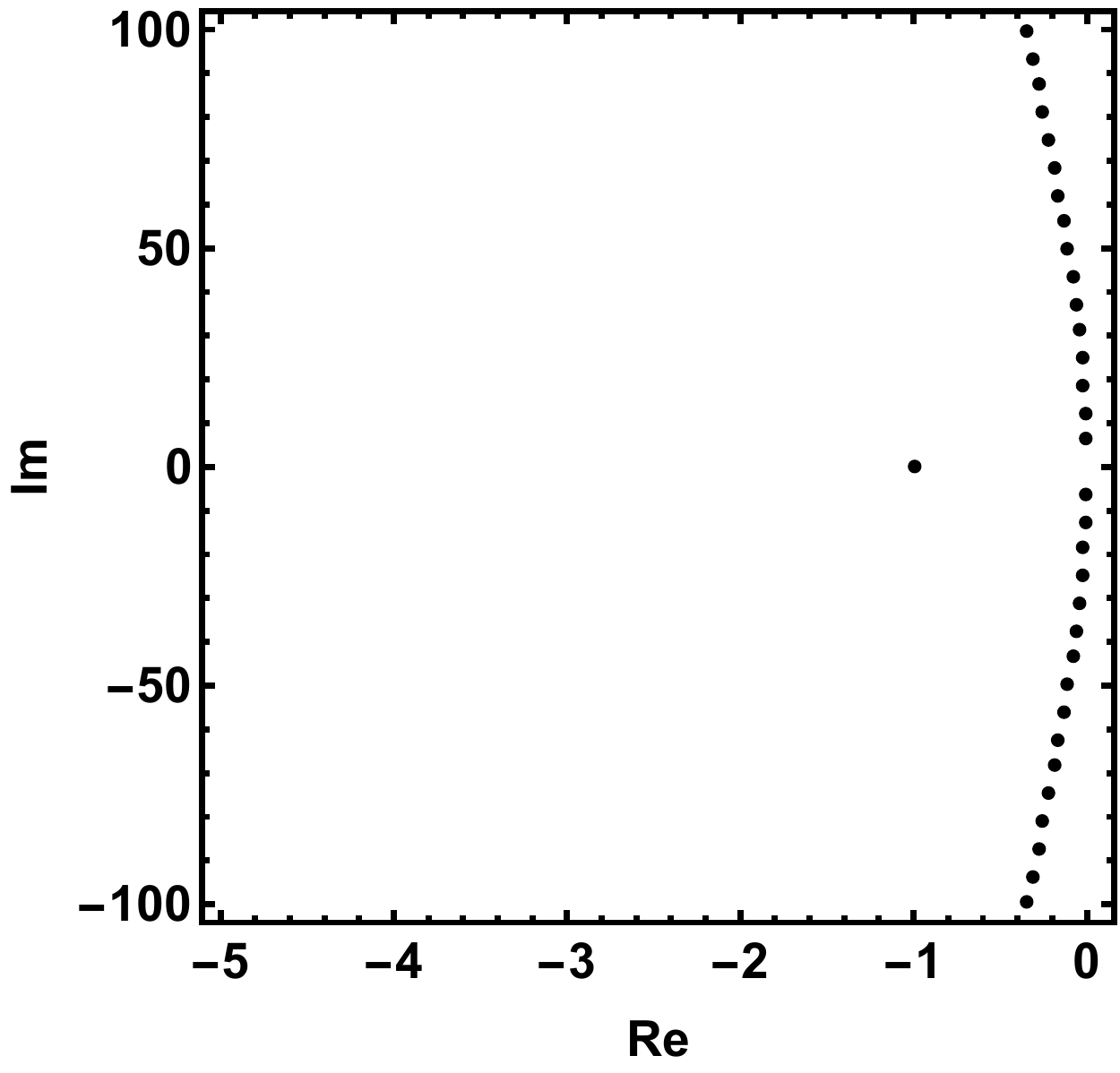}
	\end{center}
	\caption{Characteristic roots of \eqref{charint} on the complex plane for $\rho=1,10,100$. As $\rho \to \infty$, characteristic roots converge towards the imaginary axis.}
\label{spectrumfig}
\end{figure}

%====================================================================

\section{Discussion}

The delayed logistic equation has received much attention in past decades in the analysis of nonlinear delay differential equations. However, its biological validity has been questioned, despite the fact that it was introduced by Hutchinson to explain observations of  oscillatory behaviour in ecological systems. Here, we introduced a new logistic-type delay differential equation, derived from the go-or-grow hypothesis which has been observed for some types of cancer cells. The equation naturally includes a product of delayed and non-delayed terms, as well as both discrete and distributed delays. Detailed mathematical analysis reveals that all non-zero solutions converge to the stable positive equilibrium. For an alternative formulation of the delayed logistic equation~\cite{arino}, the authors also concluded that, as opposed to Hutchinson's equation, sustained oscillations were not possible and solutions settled at an equilibrium. In some sense the dynamical behaviour of our model is in-between these two extremes. While we have proved that all positive solutions are attracted to the positive equilibrium, for a range of parameters the convergence is so slow that, on a biologically realistic time scale, solutions may appear periodic. These long-lasting transient patterns can have various shapes, as shown in Fig. \ref{figure:solutions}, and these shapes also change in time (see Fig. \ref{figure:transient}).

%====================================================================

\begin{figure}
	\begin{center}
		\includegraphics[scale=0.28]{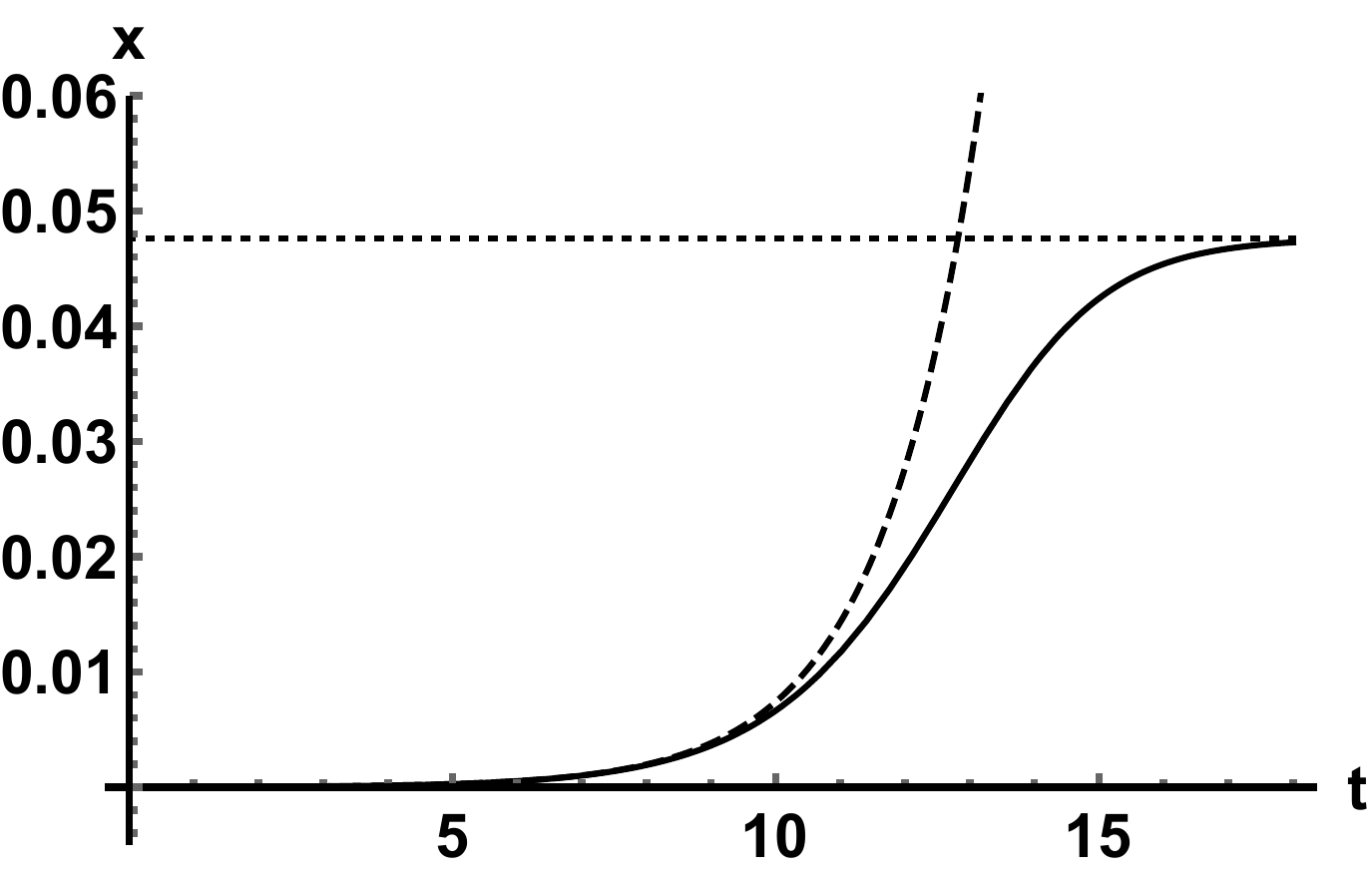} \,
		\includegraphics[scale=0.28]{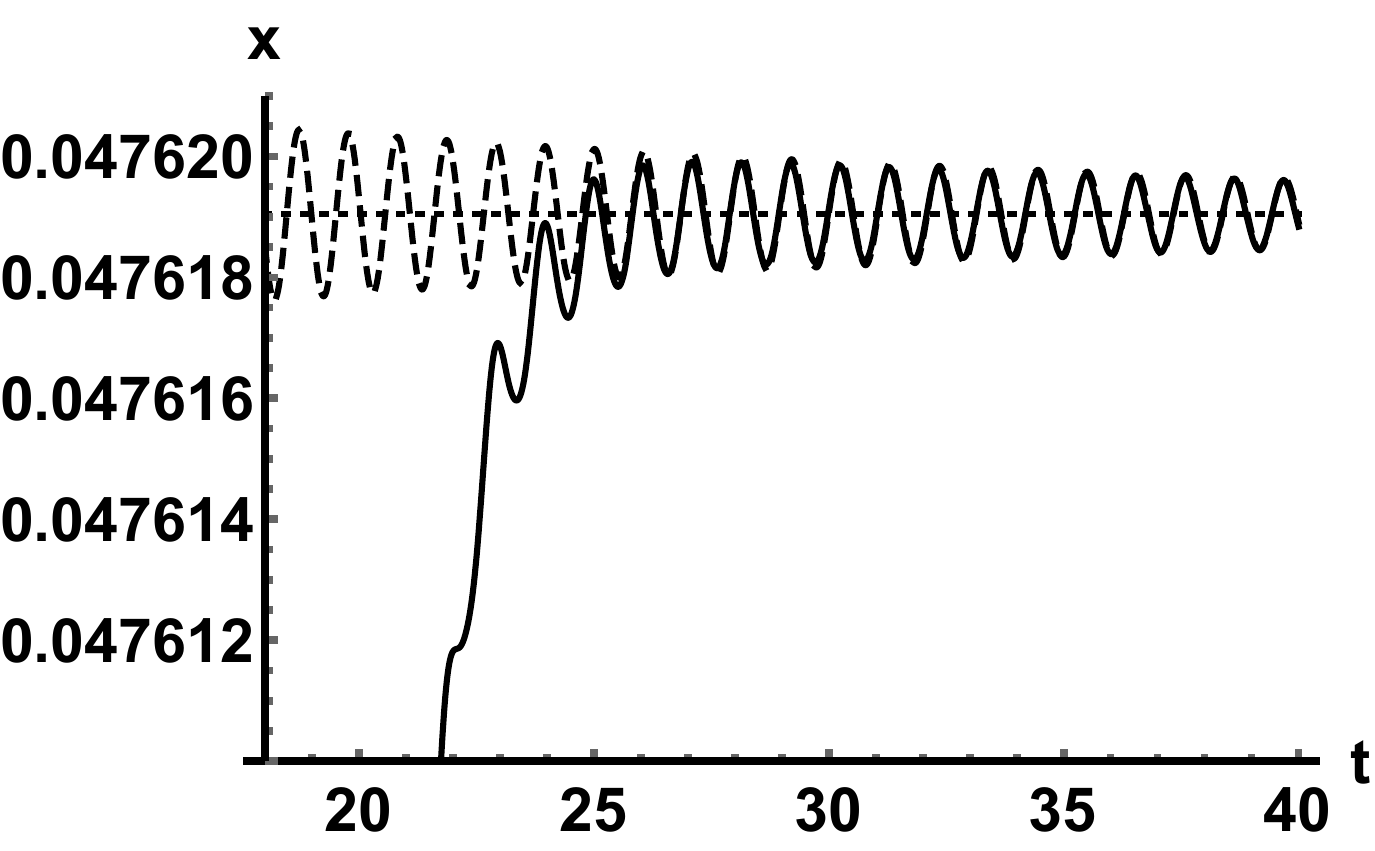} \,
		\includegraphics[scale=0.28]{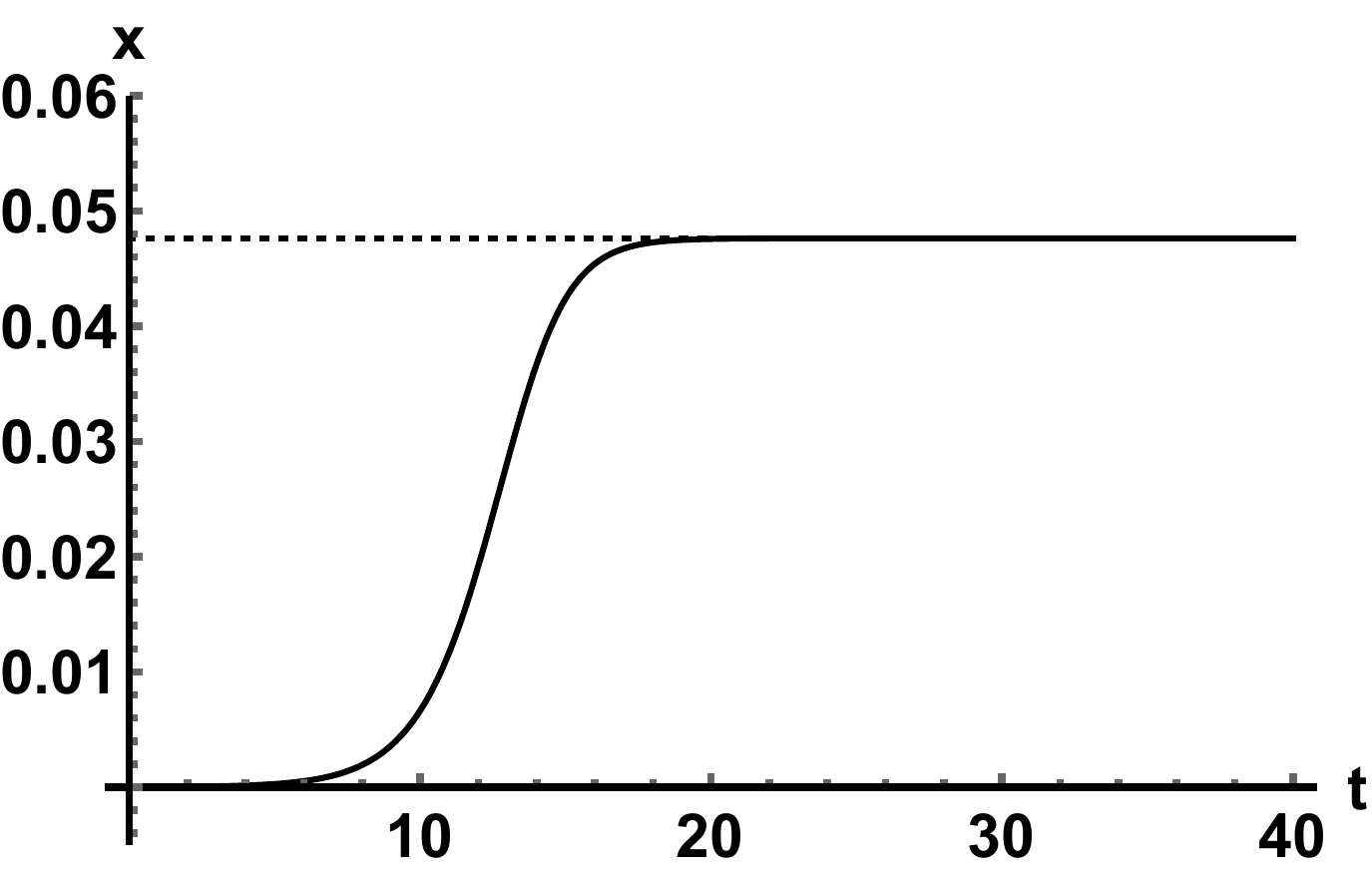}
	\end{center}
	\caption{Illustration of the unique positive heteroclinic orbit for $\rho=20$. Left: the initial phase of the solution (solid) is plotted alongside the exponential eigenfunction (dashed) corresponding to the leading real eigenvalue at zero. Center: we zoomed in to see the late part of the solution (solid), aligning nicely with the oscillatory pattern corresponding to the leading pair of complex eigenvalues at the positive equilbrium (dashed). Right: the heteroclinic solution is shown combining the two time scales. The dotted line is the positive equilbrium in all three figures.}
	\label{figure:hetero}
\end{figure}

We have also fully described the global attractor as the union of the two equilibria and a unique connecting orbit. This heteroclinic solution is depicted in Fig. \ref{figure:hetero} when $\rho=20$. In this case, the leading real eigenvalue at zero is $\lambda\approx 0.66$, while the leading pair of eigenvalues at the positive equilibrium are $\lambda \approx -0.04 \pm 6i$. We can numerically observe how the solution is aligned to the leading eigenspaces of the linearizations near the two equilibria.

Our model is based on the commonly-invoked modelling assumption that proliferating cells abort proliferation when placement of a daughter cell is not possible due to spatial crowding constraints. However, there are other potential models of crowding-limited proliferation that could be encoded within the same framework. For example, one could assume that, instead of aborting the proliferation attempt, proliferating cells instead enter a waiting state until space for the daughter cell becomes available, or that, in order to enter the proliferating state, cells must be able to ``reserve'' a site for the daughter cell to be placed into. These, and other possible biological hypotheses constitute a family of different logistic-type models, the behaviours of which we will systematically compare in subsequent works. 

As far as we are aware, this work represents the first step to understand the go-or-grow mechanism when a delay caused by the cell cycle length is explicitly incorporated as a biological parameter. Future work will explore more accurate models of the effects of spatial crowding upon proliferation, using both mean-field and moment dynamics models. In the transformed equation \eqref{eq}, the parameter $\rho$ is the product of the time delay, $\tau$, and the proliferation rate, $r$. As such, increasing either of those two parameters will have a similar effect on the dynamics (although on different time scales in terms of the original, dimensional equation). In models that include spatial details, however, there is an additional key parameter (that is ignored by equation~\eqref{eq}), namely the cell motility rate. An important future goal is to understand how the speed of \textit{e.g.} cancer invasion depends on cell-level parameters such as cell motility and proliferation rates, and the cell cycle length. To this end, spatial models of the go-or-grow mechanism that incorporate cell cycle delays need to be developed. Travelling wave solutions of reaction-diffusion systems with delays are closely related to the heteroclinic orbits of the reaction systems (see for example \cite{faria2}). As such, the results we provide on the unique heteroclinic orbit connecting the zero and the positive equilibria will help shed light on invasion speeds in these cases.

\section*{Acknowledgments} GR was supported by NKFI FK 124016 and MSCA-IF 748193. REB is a Royal Society Wolfson Research Merit Award holder and would like to thank the Leverhulme Trust for a Research Fellowship.

\end{document}